\documentclass[11pt]{amsart}
\usepackage[T1]{fontenc}
\usepackage[utf8]{inputenc}
\usepackage{lmodern}
\usepackage[margin=1in]{geometry}
\usepackage{amssymb,mathtools,mathrsfs}
\usepackage[expansion=false]{microtype}
\usepackage{enumitem,booktabs,array,needspace}
\usepackage{xcolor,graphicx}
\usepackage[colorlinks=true,linkcolor=blue,citecolor=blue,urlcolor=blue]{hyperref}
\newtheorem{theorem}{Theorem}[section]
\newtheorem*{introtheorem}{Theorem}
\newtheorem{proposition}[theorem]{Proposition}
\newtheorem{lemma}[theorem]{Lemma}
\newtheorem{corollary}[theorem]{Corollary}
\theoremstyle{definition}
\newtheorem{definition}[theorem]{Definition}

\newtheorem*{unnumberedremark}{Remark}
\theoremstyle{remark}

\numberwithin{equation}{section}
\allowdisplaybreaks
\title[Conifold factorization and GW theory]{Conifold factorization in topological recursion\\
and Gromov--Witten theory}
\author{Juping Chen}
\address{Juping Chen, Department of Mathematical Sciences, Tsinghua University, Haidian District, Beijing 100084, China}
\email{cjp24@mails.tsinghua.edu.cn}
\author{Linji Chen}
\address{Linji Chen, Beijing International Center for Mathematical Research, Peking University, 5 Yiheyuan Road, Beijing 100871, China}
\email{aco@stu.pku.edu.cn}
\author{Bohan Fang}
\address{Bohan Fang, Beijing International Center for Mathematical Research, Peking University, 5 Yiheyuan Road, Beijing 100871, China}
\email{bohanfang@gmail.com}
\author{Zhengyu Zong}
\address{Zhengyu Zong, Department of Mathematical Sciences, Tsinghua University, Haidian District, Beijing 100084, China}
\email{zyzong@mail.tsinghua.edu.cn}
\begin{document}
\begin{abstract}
We prove a factorization theorem for ordinary topological recursion near a regular nodal degeneration of a spectral curve, allowing logarithmic spectral coordinates. The partition function for closed genera $g\ge2$ factors into the partition function of the normalization, a universal Gaussian vacuum, and the exponential of a connected graph sum. This graph sum has strictly positive order in the vanishing period, and thus the factorization recovers the conifold gap, identifies the term of degree zero in the vanishing period with the free energy of the normalization, and gives finite graph formulas for the coefficients of positive powers of the period. At fixed genus, the regular series converges jointly in the period and the parameters along the nodal locus. The Gaussian neck tensors can be expressed in terms of relative Gromov--Witten invariants of a parametrized $\mathbb P^1$. As examples, we use remodeling to obtain Gromov--Witten factorizations for local $\mathbb P^2$ and local $\mathbb F_0$ in the conifold frame, with backgrounds $\mathbb C^3$ and the resolved conifold, respectively. The analogous local $\mathbb F_1$ factorization has background $\mathcal O(1)\oplus\mathcal O(-3)$ over $\mathbb P^1$.
\end{abstract}
\maketitle
\section{Introduction}
\label{sec:introduction}

\subsection{Background and motivation}
\label{subsec:intro-background}

Mirror symmetry relates the Gromov--Witten theory of a Calabi--Yau threefold to the geometry of its mirror. A regular conifold degeneration singles out a vanishing cycle and a natural expansion of the higher-genus free energies in its period. This expansion reflects both the universal local behavior near the collapsing cycle and its coupling to the remaining geometry. We study how these contributions combine, with particular attention to the geometric meaning of the regular part.

The physical interpretation of a conifold singularity in terms of massless states was developed by Strominger \cite{Strominger}. Ghoshal--Vafa \cite{GhoshalVafa} related the topological string near the conifold to the $c=1$ string at the self-dual radius, identifying a universal contribution to the higher-genus free energies. In the conventions of this paper, the expected behavior in the vanishing period $t$ is
\begin{equation}\label{eq:intro-gap}
 F_g(t)=\frac{B_{2g}}{2g(2g-2)}t^{2-2g}+O(1),\qquad g\geq2,
\end{equation}
where the Bernoulli numbers are defined by $z/(e^z-1)=\sum_{j\geq0}B_jz^j/j!$. The absence of the intermediate polar terms $t^{3-2g},\ldots,t^{-1}$ is called the \emph{conifold gap}. The holomorphic anomaly equations of Bershadsky--Cecotti--Ooguri--Vafa (BCOV) \cite{BCOV} give recursive relations among the higher-genus B-model amplitudes. Huang--Klemm \cite{HK} and Huang--Klemm--Quackenbush \cite{HKQ} used the gap to constrain the holomorphic ambiguities in their integration. The gap specifies the singular part of $F_g$. We study its finite part and positive-power coefficients.

On the mathematical side, the gap for local $\mathbb P^1$, namely the resolved conifold $\operatorname{Tot}_{\mathbb P^1}(\mathcal O(-1)\oplus\mathcal O(-1))$, follows from the multiple-cover formula of Faber--Pandharipande \cite[Theorem~3]{FP}. Explicit analytic-continuation formulas separating the universal pole from the regular terms appear in Pasquetti--Schiappa \cite[Equations~(4.17)--(4.18)]{PS} and Alim \cite[Theorem~4.6]{Alim}. For local $\mathbb P^2$, Coates--Iritani \cite{CI} established the conifold pole-order bound and verified the gap through genus seven. Brini \cite{Brini} subsequently proved the gap in all genera by identifying, up to an additive constant, the conifold Gromov--Witten potentials with the large-rank free energies of a convergent statistical ensemble of repulsive particles; this also completes the all-genus BCOV reconstruction for local $\mathbb P^2$. For semi-projective toric Calabi--Yau $3$-folds, \cite{CCFZ} and the remodeling theorem \cite{EO2015,FLZ} imply the gap in the conifold frame at regular one-node degenerations. More recently, Chang--Guo--You--Zhang \cite{CGYZ} proved the all-genus conifold gap for the quintic threefold, the one-parameter Fermat Calabi--Yau hypersurfaces of degrees $6$, $8$, and $10$, and local $\mathbb P^2$. Their proof uses mixed spin $P$-field theory and analogous master-space constructions to transfer a universal gap from a cone-vertex theory to Gromov--Witten theory. The cone-vertex theory is described as a translation of the cohomological field theory of a point by means of a formal Laplace transform. The same method also proves the gap for the associated Fan--Jarvis--Ruan--Witten theories.

For toric Calabi--Yau $3$-folds and orbifolds, mirror symmetry admits a formulation in terms of a mirror curve. Eynard--Orantin \cite{EO2007} introduced topological recursion, which associates correlation differentials and free energies to spectral data on a Riemann surface. Bouchard--Klemm--Mari\~no--Pasquetti (BKMP) \cite{BKMP,BKMPorb} proposed that this recursion on a framed mirror curve computes the open and closed Gromov--Witten potentials of the toric Calabi-Yau 3-folds. Eynard--Orantin \cite{EO2015} provided a proof of the Remodeling Conjecture for smooth semi-projective toric Calabi--Yau $3$-folds; Fang--Liu--Zong \cite{FLZ} proved it for general semi-projective toric Calabi--Yau $3$-orbifolds. The relation with the BCOV formalism is also reflected in the holomorphic anomaly equations for the non-holomorphic amplitudes constructed from spectral curves \cite{EMO}.

Mari\~no--Zakany \cite{MZ} and Kashaev--Mari\~no--Zakany \cite{KMZ} studied conifold expansions using matrix models obtained from quantized mirror curves. Using the spectral-theory/topological-string correspondence proposed by Grassi--Hatsuda--Mari\~no \cite{GHM}, their perturbative expansions around the Gaussian model compute regular coefficients and predict the finite constants for local $\mathbb P^2$ and local $\mathbb F_0$.

The behavior of topological recursion under singular limits was already studied in \cite{EO2007}. Iwaki--Kidwai \cite{IK} obtained exact free energies for hypergeometric-type curves, and Iorgov--Iwaki--Lisovyy--Zhuravlov \cite[Theorem~4.9]{IILZ} proved the conifold gap for the Painlev\'e~I elliptic family. More recently, the conifold-gap theorem of \cite{CCFZ} established the universal polar term and the joint holomorphy of the remainder for regular one-node degenerations of toric mirror curves. That result treats both separating and nonseparating nodes and allows general family deformations with spectator parameters along the nodal locus. We seek a geometric description of the holomorphic remainder.

Eynard's intersection formula \cite[Theorem~4.1]{EynardIntersection} expresses recursion invariants through local theories near the ramification points, with independently chosen bidifferentials. The differences between these local bidifferentials and the global one determine the edge weights connecting the local theories. Its proof groups connected local subgraphs into vertex correlators \cite[Appendix~A.3]{EynardIntersection}, an operation we use to organize the background and neck contributions.

In this paper, we prove a factorization into the theory of the normalization, a universal Gaussian vacuum, and an interaction determined by the germ of the spectral family. Here the background is the normalization with its inherited spectral functions and polarization, and a neck joins the node preimages. Every interaction has strictly positive order in the vanishing period. Thus the background determines the finite part, and finite connected graph sums determine all positive-power coefficients. As applications, we work out Gromov--Witten factorizations with a relative-$\mathbb P^1$ neck for local $\mathbb P^2$, local $\mathbb F_0$, and local $\mathbb F_1$.

\subsection{Statement of the main results}
\label{subsec:intro-results}

Our main result concerns ordinary topological recursion near a regular spectral node. Let $\mathcal S_{q,s}=(\Sigma_{q,s},x_{q,s},y_{q,s},B_{q,s})$ be a regular spectral-node family in the sense of Definition~\ref{def:regular-node}. Here $q$ is transverse to the nodal divisor and $s$ denotes the spectator parameters in a polydisc centered at $s_0$. The central fiber has one ordinary node; on its two branches, $dx$ is nonzero and the slopes $dy/dx$ are distinct. The other ramification points extend as disjoint simple zeros of $dx$, away from the node and punctures, with $dy\ne0$. The spectral coordinates may be meromorphic functions or local logarithms of meromorphic functions, as specified in Section~\ref{sec:setup}. The fundamental bidifferential $B_{q,s}$ is normalized in an adapted polarization, with zero period around the neck. For an oriented loop $\gamma$ around the neck, set
\[
 t=\frac{1}{2\pi i}\int_\gamma y_{q,s}\,dx_{q,s}.
\]
We require $t=q\,c(q,s)$ with $c$ holomorphic and $c(0,s)\neq0$, so that $(t,s)$ are local parameters.

Let $\nu:\widetilde\Sigma_{0,s}\to\Sigma_{0,s}$ be the normalization. The background spectral curve is
\[
 \mathcal B_s=
 \bigl(\widetilde\Sigma_{0,s},\nu^*x_{0,s},\nu^*y_{0,s},
 B^{\mathrm{back}}_s\bigr),
\]
where $B^{\mathrm{back}}_s$ has the inherited period normalization and is zero between distinct connected components. For a disconnected normalization, free energies are added and partition functions are multiplied over its components. We use the stable partition functions
\[
 Z(\mathcal S;\hbar)=\exp\!\left(\sum_{g\geq2}
 \hbar^{2g-2}F_g(\mathcal S)\right),\qquad
 Z_G(t;\hbar)=\exp\!\left(\sum_{g\geq2}
 \hbar^{2g-2}C_gt^{2-2g}\right),
 \quad C_g=\frac{B_{2g}}{2g(2g-2)}.
\]
Here $\hbar$ is formal; closed genera zero and one are omitted.

\begin{introtheorem}[Normalization and neck factorization; Theorem~\ref{thm:factorization}]
For a regular spectral-node family,
\begin{equation}\label{eq:intro-factorization}
 Z(\mathcal S_{t,s};\hbar)
 =Z(\mathcal B_s;\hbar)\,Z_G(t;\hbar)
   \exp\mathcal I(t,s;\hbar).
\end{equation}
The interaction is an explicit connected stable graph sum of the form
\[
 \mathcal I(t,s;\hbar)
 =\sum_{g\geq2}\hbar^{2g-2}\sum_{k\geq1}I_{g,k}(s)t^k.
\]
For every fixed $g$ and $k$, the coefficient $I_{g,k}$ is a finite sum over graphs and index assignments, determined by finitely many background tensor components and Taylor coefficients of the source and edge weights. Equivalently, for every $g\geq2$,
\begin{equation}\label{eq:intro-free-energy}
 F_g(\mathcal S_{t,s})
 =C_gt^{2-2g}+F_g(\mathcal B_s)
  +\sum_{k\geq1}I_{g,k}(s)t^k.
\end{equation}
The last series converges jointly in $(t,s)$ at fixed genus and belongs to $t\mathbb C\{t,s-s_0\}$.
\end{introtheorem}

The theorem includes separating and nonseparating nodes with backgrounds of arbitrary genus. It gives the conifold gap and identifies the finite part:
\[
 \lim_{t\to0}\bigl(F_g(\mathcal S_{t,s})-C_gt^{2-2g}\bigr)
 =F_g(\mathcal B_s).
\]
The factor $Z_G$ is the universal Gaussian vacuum. The neck factor $Z_G\exp\mathcal I$ also includes the coupling to the background and depends on the germ of the spectral family.

The vertices of the graph sum are background correlators and Gaussian residue moments. The latter are finite sums over pairings, and every stable moment with at least one insertion is divisible by $t$ (Proposition~\ref{prop:gaussian-moments}). Each source leaf carries a coefficient $-d_I$ of the one-form difference from the separated background and Gaussian neck; each edge carries a coefficient $k_{IJ}$ of the bidifferential difference, computed by sewing \cite{Yamada,GKN,HuNorton,CCFZ}. On the neck, $d_{\pm,1}$ and $d_{\pm,2}$ are $O(t)$. Integrating the variation formulas of Eynard--Orantin and Kazarian \cite{EO2007,ABDKS}, with contour integrals over the neck boundary and around the background ramification points, gives the connected graph expansion.

The source sums vanish at genus-zero vertices with at most two other half-edges (Lemma~\ref{sew:source-cancellation}). This cancellation, the positive $t$-order of the Gaussian moments, and the $O(t)$ background sources and edges give finite graph formulas at each genus and order in $t$, with zero constant interaction term. The Gaussian free energies are given by the Bernoulli formula in Proposition~\ref{prop:gaussian-vacuum}. A fixed-genus residue argument identifies the graph coefficients with the convergent expansion of the free energy.

The Gaussian neck tensors also admit a Gromov--Witten interpretation. By the correspondence between Hurwitz numbers and relative Gromov--Witten invariants \cite{OP}, the Gaussian pairing sums can be written as finite sums of connected relative invariants of a parametrized $\mathbb P^1$ with three relative points. The resulting stable neck tensors have positive $t$-order (Proposition~\ref{prop:gw-neck-identification}). The Gaussian vacuum is the scaling limit of the positive-degree resolved-conifold potential given by the multiple-cover formula \cite{FP}.

We illustrate the factorization through the local $\mathbb P^2$, local $\mathbb F_0$, and local $\mathbb F_1$ geometries. These examples give explicit background contributions coupled to the same universal neck. For a smooth surface $S$, write $K_S$ also for the total space of its canonical bundle, and put
\[
 R=\operatorname{Tot}_{\mathbb P^1}
       \bigl(\mathcal O(-1)\oplus\mathcal O(-1)\bigr).
\]
The backgrounds for $K_{\mathbb P^2}$ and $K_{\mathbb F_0}$ are $\mathbb C^3$ and $R$, respectively. Their stable closed couplings are $g_s/3$ and $g_s/2$. The background tensors are contour integrals of the continued framed open Gromov--Witten potentials in Section~\ref{sec:gw}. In these statements, the superscript $c$ denotes the conifold calibration of Subsection~\ref{subsec:gw-conventions}: analytic continuation from large radius along a fixed path, followed by the prescribed change to the conifold polarization. The spectral period and coupling are converted by $t_A=i t$ and $\hbar=i g_s$. We use the superscript $A$ for Gromov--Witten potentials.

Set
\[
 A_g=\frac{(-1)^g|B_{2g}B_{2g-2}|}
 {2(2g)(2g-2)(2g-2)!},\qquad
 a_g=\frac{|B_{2g}|}{2g(2g-2)!}.
\]
For spectator parameters near background large radius, with the continuation and framings fixed in Section~\ref{sec:gw}, Corollary~\ref{gw:toric-factorization} gives, for $g\geq2$,
\begin{align}
 F_g^{K_{\mathbb P^2},c,A}(t_A)
 &=C_gt_A^{2-2g}+3^{2-2g}A_g
   +\sum_{k\geq1}I^{\mathbb P^2,A}_{g,k}t_A^k,
 \label{eq:intro-local-p2}\\
 F_g^{K_{\mathbb F_0},c,A}(t_A,r)
 &=C_gt_A^{2-2g}
   +2^{2-2g}\bigl(2A_g+a_g\operatorname{Li}_{3-2g}(-r)\bigr)
   +\sum_{k\geq1}I^{\mathbb F_0,A}_{g,k}(r)t_A^k,
 \label{eq:intro-local-f0}
\end{align}
where $0<|1/r|$ is sufficiently small. The degree-zero terms are included. Every coefficient is determined by the connected contraction of the background open potentials with the universal relative-$\mathbb P^1$ neck, using the source and edge coefficients of the spectral family. The remainders are holomorphic at $t_A=0$, jointly in the regular spectator parameters. These identities follow from the spectral factorization and the remodeling theorem \cite{FLZ}. These expansions are in parallel with the matrix-model statements of Mari\~no--Zakany and Kashaev--Mari\~no--Zakany \cite{MZ,KMZ}.

For $K_{\mathbb F_1}$, the background is
\[
 X_2=\operatorname{Tot}_{\mathbb P^1}
          \bigl(\mathcal O(1)\oplus\mathcal O(-3)\bigr),
 \qquad Q=-q(1-q)^3,
\]
with coupling $g_s$ and the equivariant weights and framing specified in Subsection~\ref{subsec:gw-backgrounds}. Its closed theory is defined by localization, and its open theory by formal-relative Gromov--Witten theory or open localization \cite{LLLZ,FangLiu}. Mari\~no \cite{MarinoLocalCurves} proposed the spectral description, Eynard \cite{EynardPartitions} proved the closed correspondence, and Guo--Xu--Zhang \cite{GXZ} established the open and closed comparison. Proposition~\ref{gw:f1} gives the comparison in our conventions. Corollary~\ref{gw:f1-factorization} then gives, for sufficiently small nonzero values of the spectator parameter $q$,
\begin{equation}\label{eq:intro-local-f1}
 F_g^{K_{\mathbb F_1},c,A}(t_A,q)
 =C_gt_A^{2-2g}
  +F_g^{X_2,A}\bigl(-q(1-q)^3\bigr)
  +\sum_{k\geq1}I^{\mathbb F_1,A}_{g,k}(q)t_A^k,
 \qquad g\geq2.
\end{equation}

\begin{unnumberedremark}
This paper is a sequel to \cite{CCFZ}, which established the conifold gap and joint holomorphy of the remainder for regular toric one-node degenerations. Building on the normal-form and sewing analysis developed there, we identify the finite part with the free energy of the normalized background and express the coefficients of positive powers of the vanishing period as finite connected graph sums. The resulting factorization also recovers the gap.
\end{unnumberedremark}

\subsection{Outline of the paper}
\label{subsec:intro-outline}

In Section~\ref{sec:setup}, we specify the spectral data, the regularity and polarization assumptions, and the normalized background. In Subsection~\ref{sec:gaussian}, we compute the Gaussian residue moments and recall its free-energy formula. In Subsection~\ref{sec:sewing}, we construct the local normal form, normalized sewing, and contour integrals that determine the sources and edges. Section~\ref{sec:proof} integrates the relative variation formulas into a connected stable graph sum, proves positive period order, finiteness at each genus and period order, and analytic convergence, and concludes with the factorization theorem. Section~\ref{sec:gw} illustrates the theorem through three Gromov--Witten examples. It identifies their normalized backgrounds and the universal relative-GW neck, states the three factorizations, and identifies the background contour insertions with closed GW ancestors.

\subsection{Use of artificial intelligence}
\label{subsec:intro-ai-disclosure}

This work involved extensive, multi-round interaction with generative AI. Codex and Claude assisted with the writing. The authors verified the mathematical arguments and the final text.

\subsection*{Acknowledgements}

The work of Bohan Fang is partially supported by the National Key R\&D Program of China 2023YFA1009803, NSFC 12125101, and NSFC 11890661. The work of Zhengyu Zong is partially supported by NSFC grant No.~11701315 and the Natural Science Foundation of Beijing, China grant No.~1252008.

\section{Spectral data and sewing}
\label{sec:setup}
\subsection{Conventions and hypotheses}
\label{subsec:spectral-conventions}

For spectral data $(\Sigma,x,y,B)$, set $\lambda=y\,dx$. We use Eynard--Orantin recursion \cite{EO2007}, with residues at the zeros of $dx$ in $\Sigma$ and $\omega_{0,2}=B$. Its correlators $\omega_{g,n}$ are symmetric meromorphic differentials. A type $(g,n)$ is stable when $2g-2+n>0$. The recursion kernel and the genus-$g$ free energy are
\begin{align}
 K_a(p,z)&=\frac{\displaystyle\frac12\int_z^{\sigma_a(z)}B(p,\cdot)}
 {(y(z)-y(\sigma_a(z)))\,dx(z)},\label{eq:recursion-kernel}\\
 F_g&=\frac1{2g-2}\sum_{a:dx(a)=0}
 \operatorname{Res}_{p=a}\Phi_a(p)\omega_{g,1}(p),
 \qquad d\Phi_a=\lambda,\quad g\geq2.\label{eq:closed-convention}
\end{align}
Here $\sigma_a$ is the local involution at $a$, and $\Phi_a$ is a local primitive of $\lambda$. The stable recursion omits factors of type $(0,1)$. We use the signed input $\omega_{0,1}=-\lambda$: in terms of this input, the same kernel integrates from $\sigma_a(z)$ to $z$ and divides by $\omega_{0,1}(z)-\sigma_a^*\omega_{0,1}(z)$. Our free energy satisfies $F_g=-F^{(g)}_{\mathrm{EO07}}$ in the notation of \cite[Definition~4.3]{EO2007}. Set
\begin{equation}\label{eq:stable-partition}
 Z=\exp\sum_{g\geq2}\hbar^{2g-2}F_g,\qquad
 C_g=\frac{B_{2g}}{2g(2g-2)},\qquad
 Z_G(t;\hbar)=\exp\sum_{g\geq2}\hbar^{2g-2}C_gt^{2-2g},
\end{equation}
where $z/(e^z-1)=\sum B_jz^j/j!$ and $\hbar$ is formal.

We use local parameters $(q,s)\in\Delta\times S$, where $\Delta$ is a disc about $0$ and $S$ is a polydisc about $s_0$ with coordinates $s=(s_1,\ldots,s_r)$.

Let $\overline{\mathscr C}\to\Delta\times S$ be a proper flat holomorphic family of connected reduced compact curves. Its fibers $\overline\Sigma_{q,s}$ are smooth for $q\ne0$, and each $\overline\Sigma_{0,s}$ has exactly one ordinary node and is otherwise smooth. Let $\Sigma_{q,s}$ be the complement of finitely many pairwise disjoint holomorphic puncture sections, all disjoint from the node. The nodal locus is the divisor $q=0$. Following \cite{CCFZ}, we call $s$ the spectator parameters; they describe the variation of the nodal curve and its remaining spectral data.

Each of $x_{q,s}$ and $y_{q,s}$ is either a meromorphic function on $\overline\Sigma_{q,s}$ that is regular on $\Sigma_{q,s}$, or a local logarithm of such a function that is nowhere zero on $\Sigma_{q,s}$. The underlying meromorphic functions vary holomorphically with the family away from the puncture sections. We call $x,y$ the spectral coordinates and set $\lambda_{q,s}=y_{q,s}\,dx_{q,s}$. In the logarithmic case, we choose local holomorphic branches near the node and every ramification point of the spectral $x$-map that extend holomorphically with the family across $q=0$. On each smooth compactified fiber, let $B_{q,s}$ be a fundamental bidifferential.

The node is nonseparating if the normalization of $\overline\Sigma_{0,s}$ is connected, and separating if that normalization has two connected components. For $q\ne0$ small, the smoothing of a neighborhood of the node is an annulus, which we call the neck. Fix an orientation of its core circle $\gamma$, the vanishing cycle.
\begin{definition}\label{def:regular-node}
We call $\mathcal S_{q,s}=(\Sigma_{q,s},x_{q,s},y_{q,s},B_{q,s})$ a regular spectral-node family if it satisfies the following conditions.
\begin{enumerate}[label=\textup{(\roman*)},ref=\roman*]
\item\label{cond:node-branches} Let $p_+,p_-$ be the node preimages on the central normalization. The local branches of $x,y$ are chosen consistently near the node, and their pullbacks have common values at $p_+,p_-$. On both branches the $x$-differential is nonzero, and the slopes are distinct:
\[
 dx(p_\pm)\ne0,\qquad
 \left.\frac{dy}{dx}\right|_{p_+}\ne
 \left.\frac{dy}{dx}\right|_{p_-}.
\]
For $q\ne0$ sufficiently small, these node conditions give exactly two simple ramification points of $x$ tending to the node as $q\to0$, with $dy\ne0$ at both. We call these points colliding and all other ramification points non-colliding.
\item\label{cond:noncolliding-ramification} The non-colliding ramification points extend over $\Delta\times S$ as disjoint holomorphic sections, away from the node and punctures. They are simple zeros of $dx$, with $dy\ne0$, including at $q=0$.
\item\label{cond:period-normalization} Choose symplectic bases on the connected components of the compactified normalization, represented away from the node preimages, and transport their $A$-cycles to nearby smooth fibers through the complement of the neck. In the nonseparating case, take $\gamma$ as $A_1$ and adjoin these transported cycles. Their integral span is the $A$-polarization adapted to the node. The cycles are transported locally constantly with the family, and $B_{q,s}$ has zero periods along them.
\item\label{cond:transverse-period} The period of $\lambda_{q,s}$ extends holomorphically to $q=0$ and satisfies
  \begin{equation}\label{eq:exact-period}
   t=\frac1{2\pi i}\int_\gamma\lambda_{q,s},
   \qquad t(0,s)=0,\qquad \partial_qt(0,s)\ne0.
  \end{equation}
\end{enumerate}
\end{definition}

By condition~\textup{(\ref{cond:transverse-period})}, after shrinking $\Delta$ and $S$, the period $t$ can be used in place of $q$, so we may use $(t,s)$ as parameters.

For example, for $v^2=u^2-q$, $x=u$, $y=u^2v$, the two branches at $q=0$ have the same slope $dy/dx=0$ at the node, so condition~\textup{(\ref{cond:node-branches})} fails. The period is $q^2/8$ when $\gamma$ is clockwise in $u+v$, so condition~\textup{(\ref{cond:transverse-period})} also fails although the smoothing is transverse.

\begin{definition}[Background]\label{def:background}
Let $\nu:\widetilde\Sigma_{0,s}\to\Sigma_{0,s}$ be the normalization. The background spectral curve is
\begin{equation}\label{eq:background}
 \mathcal B_s=
 (\widetilde\Sigma_{0,s},\nu^*x_{0,s},\nu^*y_{0,s},
       B_s^{\mathrm{back}}),
\end{equation}
where $B_s^{\mathrm{back}}$ is the fundamental bidifferential on the compactified normalization, normalized by the component $A$-cycles chosen in condition~\textup{(\ref{cond:period-normalization})} and set to zero between distinct components. We call these the inherited $A$-cycles. Write $\omega_{h,n}^{\mathrm{back}}=\omega_{h,n}(\mathcal B_s)$ and $Z_{\mathrm{back}}=Z(\mathcal B_s;\hbar)$. For a disconnected normalization, the partition function is the product over its connected components.
\end{definition}

The punctures of $\widetilde\Sigma_{0,s}$ are inherited from $\Sigma_{0,s}$. The two node preimages $p_+,p_-$ are smooth points of $\widetilde\Sigma_{0,s}$, where the inherited $x,y$ have holomorphic local branches and $dx(p_\pm)\ne0$ by condition~\textup{(\ref{cond:node-branches})}.

\subsection{The universal Gaussian neck}\label{sec:gaussian}

For $t\ne0$, define the universal neck $G_t$ to be the spectral curve on $\mathbb P^1\setminus\{0,\infty\}$ with data
\begin{equation}\label{gau:curve}
 \begin{gathered}
 U=w,\qquad V=t/w,\qquad UV=t,\qquad
 x_G=U+V,\qquad y_G=(U-V)/2,\\
 B_G=\frac{dw_1\,dw_2}{(w_1-w_2)^2},\qquad
 \lambda_G=y_G\,dx_G
 =\tfrac12(w-2t/w+t^2/w^3)\,dw.
 \end{gathered}
\end{equation}
Fix $\rho>0$. For $0<|t|<\rho^2$, the annulus $|t|/\rho<|w|<\rho$ contains the ramification points $\pm\sqrt t$. The period $(2\pi i)^{-1}\int\lambda_G$ along any core circle oriented clockwise in $w$ equals $t$. The map $w\mapsto\sqrt t\,w$ pulls back the spectral one-form of $G_t$ to $t$ times that of $G_1$, and preserves $B_G$. The kernel~\eqref{eq:recursion-kernel} gives $\omega^{G_1}_{1,1}=w^3\,dw/(w^2-1)^4$.

\subsubsection{Gaussian residue moments}
\label{subsubsec:gaussian-moments}
For $m\ge1$, define the residue functionals $\mathcal L_m^+$ at $\infty$ and $\mathcal L_m^-$ at $0$ on meromorphic one-form germs $\omega$ by
\begin{equation}\label{gau:end-moments}
 \mathcal L_m^+\omega
 =-\operatorname{Res}_{w=\infty}\frac{U^m}{m}\omega,
 \qquad
 \mathcal L_m^-\omega
 =\operatorname{Res}_{w=0}\frac{V^m}{m}\omega.
\end{equation}
The residues use the local coordinates $1/w$ at $\infty$ and $w$ at $0$. They depend only on the respective Laurent expansions, so the same formulas apply to formal Laurent differentials. We apply these functionals one variable at a time to the Gaussian recursion differentials, leaving a differential in the remaining variables.

For a stable Gaussian correlator, write $\omega=\omega^{G_t}_{h,n}$. As a differential in $w$, $\omega$ is meromorphic on $\mathbb P^1$, with poles only at $w=\pm\sqrt t$. The involution $\iota(w)=t/w$ fixes these points, and the linear loop equation says that $\omega+\iota^*\omega$ is holomorphic there. The sum is holomorphic on $\mathbb P^1$, which has no nonzero holomorphic one-forms. Hence $\iota^*\omega=-\omega$. Since $\iota$ exchanges the punctures $0,\infty$ and $V\circ\iota=U$, invariance of residues under change of coordinate gives
\[
 \begin{aligned}
 \mathcal L_m^-\omega
 &=\operatorname{Res}_{w=0}\frac{V^m}{m}\omega
  =\operatorname{Res}_{w=\infty}\frac{U^m}{m}\,\iota^*\omega\\
 &=-\operatorname{Res}_{w=\infty}\frac{U^m}{m}\omega
  =\mathcal L_m^+\omega.
 \end{aligned}
\]

For $n>0$, $2h-2+n>0$, and $m_i\ge1$, define the Gaussian residue moments by
\[
 \mathcal A_{h,n}(\mathbf m;t)
 =\left(\prod_{i=1}^n\mathcal L_{m_i}^+\right)
   \omega^{G_t}_{h,n},\qquad \mathbf m=(m_1,\ldots,m_n).
\]
In each variable, $\mathcal L_{m_i}^+$ may be replaced by $\mathcal L_{m_i}^-$ without changing the moment. Let $\{\mathbf v_m:m\ge1\}$ be a basis of a complex vector space. For fixed $t$, define the Gaussian tensor to be the symmetric $n$-linear form determined by
\[
 \mathcal A_{h,n}(\mathbf v_{m_1},\ldots,\mathbf v_{m_n};t)
 =\mathcal A_{h,n}(m_1,\ldots,m_n;t).
\]

We express the classical Gaussian pairing expansion \cite[Section~2]{ACNP} in this residue basis. For positive integers $k_1,\ldots,k_n$ with $\sum_i k_i=2E$, fix a permutation $\sigma\in S_{2E}$ whose cycles have lengths $k_1,\ldots,k_n$. A pairing is a fixed-point-free involution $\alpha$ of the $2E$ labelled slots. Define
\begin{equation}\label{gau:pairing-count}
 N_{h,n}(\mathbf k)=\#\left\{\alpha:
 \begin{array}{l}
 \langle\sigma,\alpha\rangle\text{ acts transitively},\\
 n-E+\#\operatorname{cycles}(\sigma\alpha)=2-2h
 \end{array}\right\}.
\end{equation}
We set the count to zero when $\sum_i k_i$ is odd. For stable $(h,n)$ we also set it to zero when any $k_i=0$: the trace $\operatorname{Tr}M^0=N$ is constant and has no connected stable cumulants. Each involution contributes one labelled pairing with weight one. Write the rescaled Chebyshev polynomials as
\begin{equation}\label{gau:chebyshev}
 \begin{gathered}
 C_m(x;t)=U^m+V^m
 =\sum_{j=0}^{\lfloor m/2\rfloor}c(m,j)t^jx^{m-2j},\\
 c(m,j)=(-1)^j\frac{m}{m-j}\binom{m-j}{j},\qquad m\ge1.
 \end{gathered}
\end{equation}

\begin{proposition}[Gaussian residue moments]\label{prop:gaussian-moments}
For $n>0$, $2h-2+n>0$, and $m_i\ge1$, the residue moment $\mathcal A_{h,n}(\mathbf m;t)$ vanishes for odd $\sum_i m_i$; otherwise
\begin{equation}\label{gau:finite-moments}
 \begin{aligned}
 \mathcal A_{h,n}(\mathbf m;t)
 &=\frac{t^{\frac12\sum_i m_i+2-2h-n}}{\prod_i m_i}\\
 &\quad\times
 \sum_{\substack{0\le j_i\le\lfloor m_i/2\rfloor\\1\le i\le n}}
 \left(\prod_i c(m_i,j_i)\right)
 N_{h,n}(m_1-2j_1,\ldots,m_n-2j_n).
 \end{aligned}
\end{equation}
Every nonzero residue moment in \eqref{gau:finite-moments} is divisible by $t$.
\end{proposition}
\begin{proof}
We first express Gaussian matrix moments in terms of the pairings counted by $N_{h,n}$. Let $M$ be a Gaussian Hermitian matrix of size $N$, with entry covariance $\langle M_{ab}M_{cd}\rangle=\hbar\delta_{ad}\delta_{bc}$. Wick's formula expands the connected cumulant $\langle\prod_i\operatorname{Tr}M^{k_i}\rangle_c$ as a sum over pairings that connect all $n$ traces. The permutation $\sigma$ records the cyclic order of the entries in each trace, and $\alpha$ records the paired entries. Thus connectedness is the transitivity condition in \eqref{gau:pairing-count}. Each pair contributes $\hbar$, and each cycle of $\sigma\alpha$ gives a free matrix index and hence a factor $N$. Writing $f=\#\operatorname{cycles}(\sigma\alpha)$ and substituting $N=t/\hbar$, a genus-$h$ pairing has weight
\[
 \hbar^EN^f=\hbar^{2h-2+n}t^f,
 \qquad f=E+2-2h-n.
\]

Applying the matrix-model correlation formula of Eynard--Orantin \cite[Section~10.1.2]{EO2007} to the Gaussian potential, with the kernel~\eqref{eq:recursion-kernel}, identifies $\omega^{G_t}_{h,n}$ with the coefficient of $\hbar^{2h-2+n}$ in
\[
 \left\langle\prod_{i=1}^n\operatorname{Tr}(x_i-M)^{-1}\right\rangle_c
 \prod_{i=1}^n dx_i,\qquad x_i=x_G(w_i).
\]
Expanding $(x-M)^{-1}=\sum_{k\ge0}M^kx^{-k-1}$ shows that the negative residue at infinity against $x^k\,dx$ extracts the corresponding trace moment. Since $x_G\sim w$ at infinity, this gives
\[
 \left(\prod_{i=1}^n-\operatorname{Res}_{w_i=\infty}
 x_G(w_i)^{k_i}\right)\omega^{G_t}_{h,n}
 =N_{h,n}(\mathbf k)t^{E+2-2h-n},
 \qquad 2E=\sum_i k_i.
\]

We now pass from powers of $x_G$ to the residue functionals $\mathcal L_m^+$. A stable Gaussian differential is holomorphic at infinity, so multiplication by $V^m=t^mw^{-m}$ gives zero residue there. We may therefore replace $U^m$ by $C_m(x_G;t)=U^m+V^m$ in \eqref{gau:end-moments}. Expanding each $C_{m_i}/m_i$ by \eqref{gau:chebyshev} yields the sum in \eqref{gau:finite-moments}: choosing its $j_i$-th term sets $k_i=m_i-2j_i$ and contributes $c(m_i,j_i)t^{j_i}/m_i$. If $\sum_i m_i$ is odd, every resulting total $\sum_i k_i$ is odd, so there are no pairings.

Finally, every contributing pairing has at least one face. Its total power of $t$ is therefore
\begin{equation}\label{gau:positive-order}
 \tfrac12\sum_i m_i+2-2h-n=f+\sum_i j_i\ge1.
\end{equation}
The exponent is the same for every term in the sum. Thus every nonzero moment is divisible by $t$; if this exponent is nonpositive, all pairing counts vanish.
\end{proof}
For example, $\mathcal A_{1,1}(4;t)=t/4$ and $\mathcal A_{2,1}(8;t)=21t/8$.

\subsubsection{Gaussian free energies}
\label{subsubsec:gaussian-vacuum}
We recall the classical Gaussian free-energy formula. Its coefficient $C_g$ is the orbifold Euler characteristic $\chi_{\mathrm{orb}}(\mathcal M_g)$ computed by Harer--Zagier \cite{HZ}.
\begin{proposition}[Gaussian free energies]\label{prop:gaussian-vacuum}
For $g\ge2$,
\begin{equation}\label{gau:vacuum-result}
 F_g(G_t)=\frac1{2g-2}
 \sum_{a=\pm\sqrt t}\operatorname{Res}_{w=a}
 \Phi\,\omega^{G_t}_{g,1}
 =C_gt^{2-2g},\qquad d\Phi=\lambda_G.
\end{equation}
\end{proposition}
\begin{proof}
The Gaussian matrix-model formula appears in \cite[equation~(3.19)]{PS}. A proof for the spectral curve used here can be found in \cite[Theorem~4.9]{IKT}: the parametrization in \cite[equation~(4.2)]{IKT} agrees with $G_t$ when the parameter $\lambda$ of Iwaki--Koike--Takei equals $t$ and $w=\sqrt t\,z$. The recursion kernel and closing denominator in \cite{IKT} both have the opposite sign to ours, so the closed free energies agree.
\end{proof}
In particular, $F_2(G_t)=-1/(240t^2)$.

\Needspace{10\baselineskip}
\subsection{Sewing the spectral data}
\label{sec:sewing}

\subsubsection{The local normal form}
\label{subsubsec:local-normal-form}

We adapt the normal-form, period, and plumbing-coordinate constructions of \cite[Propositions~2.2--2.4]{CCFZ} to the coefficients used in the graph expansion. Write $\mathcal O_s$ for the ring of holomorphic germs at $s_0$, and $\mathcal O_s\{t\}$ for the ring of germs holomorphic jointly near $(t,s)=(0,s_0)$, expressed as convergent power series in $t$.

\begin{proposition}\label{sew:local-form}
For a family satisfying Definition~\ref{def:regular-node}, there are holomorphic local coordinates and a transverse parameter $\tau$ in which
\begin{equation}\label{sew:weierstrass}
 x=x_0(\tau,s)+u,\qquad v^2=u^2-\tau,\qquad
 y=y_{\mathrm{ev}}(u,\tau,s)+M(u,\tau,s)v,
 \qquad M(0,0,s)\ne0.
\end{equation}
Choose the sign of $v$ so that $\gamma$ is clockwise in the coordinate $u+v$. Write $M=\sum_{k\ge0}M_k(\tau,s)u^k$. The period $t$ is
\begin{equation}\label{sew:period-unit}
 t=\tau c(\tau,s),\qquad
 c(\tau,s)=\sum_{r\ge0}
 \frac{\operatorname{Cat}_r}{2^{2r+1}}M_{2r}(\tau,s)\tau^r,
 \qquad \operatorname{Cat}_r=\frac{1}{r+1}\binom{2r}{r}.
\end{equation}
The holomorphic function $c$ satisfies $c(0,s)=M_0(0,s)/2\ne0$. After shrinking the parameter domain, choose a holomorphic square root of $c$ and set
\begin{equation}\label{sew:period-coordinates}
 U=\sqrt c\,(u+v),\qquad V=\sqrt c\,(u-v),\qquad UV=t.
\end{equation}
On the neck, define the spectral functions
\begin{equation}\label{sew:scaled-spectral-functions}
 X=2\sqrt c\,(x-x_0)=U+V,
 \qquad Y=\frac{y-y_{\mathrm{ev}}}{2\sqrt c}.
\end{equation}
At fixed parameters, $Y\,dX=\lambda_{\mathrm{odd}}$, where
\[
 \lambda_{\mathrm{odd}}=(y-y_{\mathrm{ev}})\,dx=Mv\,du,
 \qquad \lambda_G=\frac12(U-V)^2\frac{dU}{U}.
\]
Their difference has a holomorphic primitive
\begin{equation}\label{sew:source-primitive}
 \lambda_{\mathrm{odd}}-\lambda_G=d\Psi_A,
 \qquad
 \Psi_A=\sum_{m\ge1}\frac{a_m(t,s)}{m}(U^m-V^m).
\end{equation}
The coefficients satisfy
\begin{equation}\label{sew:low-modes}
 a_m\in\mathcal O_s\{t\},\qquad a_1,a_2\in t\mathcal O_s\{t\},
 \qquad [t^r]a_m\ne0\ \Longrightarrow\ m+2r\ge3.
\end{equation}
\end{proposition}

\begin{proof}
On the normalization of the nodal fiber $q=0$, $x$ is a local coordinate near each $p_\pm$ because $dx(p_\pm)\ne0$. Write $y=f_+(x,s)$ near $p_+$ and $y=f_-(x,s)$ near $p_-$, using a common $x$-disc. Their distinct tangent slopes imply that $x-x(p_\pm)$ and $y-y(p_\pm)$ generate the maximal ideal of the central nodal fiber. The relative embedding criterion, together with flatness, then gives a local hypersurface equation for the family in $(x,y,q,s)$. Weierstrass preparation in $y$ makes this equation quadratic in $y$. Completing the square gives
\[
 (y-y_{\mathrm{ev}}(x,q,s))^2=D(x,q,s),
 \qquad D(x,0,s)=\frac14\bigl(f_+(x,s)-f_-(x,s)\bigr)^2.
\]
The difference $f_+-f_-$ has a simple zero at the node because the slopes differ. Thus $D(x,0,s)$ has a zero of order exactly two. Applying Weierstrass preparation in $x$ and completing the square again gives
\[
 D(x,q,s)=h(x,q,s)\bigl((x-x_0(q,s))^2-\tau(q,s)\bigr),
 \qquad \tau(0,s)=0,
\]
where $h$ is holomorphic and nonzero at the node. After shrinking the neighborhood, choose a holomorphic square root of $h$ and set $u=x-x_0$, $M(u,q,s)=\sqrt{h(x_0+u,q,s)}$, and $v=(y-y_{\mathrm{ev}})/M$. These coordinates give $v^2=u^2-\tau$ and $y=y_{\mathrm{ev}}+Mv$, with $M(0,0,s)\ne0$. We now regard $y_{\mathrm{ev}}$ as a function of $(u,q,s)$; all coefficients still depend on the original parameters $(q,s)$.

The one-form $y_{\mathrm{ev}}\,du$ has a holomorphic primitive $H_{\mathrm{ev}}(u,q,s)$, so its integral around the neck is zero. We may therefore compute the period using $\lambda_{\mathrm{odd}}=Mv\,du$.

Put $W=u+v$, so $u-v=\tau/W$. Then
\[
 \lambda_{\mathrm{odd}}
 =\frac14M\!\left(\frac{W+\tau/W}{2},q,s\right)(W-\tau/W)^2\frac{dW}{W},
 \qquad
 \operatorname{Res}_{W=0}u^{2r}v\,du
 =-\frac{\operatorname{Cat}_r}{2^{2r+1}}\tau^{r+1}.
\]
Odd powers of $u$ have zero residue. The displayed residues therefore give $t=\tau\widetilde c(q,s)$, where $\widetilde c$ is holomorphic and $\widetilde c(0,s)=M(0,0,s)/2\ne0$. Holomorphic dependence follows by extracting Laurent coefficients on a fixed $W$-circle. By \eqref{eq:exact-period}, $\partial_q t(0,s)\ne0$, so differentiating $t=\tau\widetilde c$ at $q=0$ gives $\partial_q\tau(0,s)\ne0$. Thus $\tau$ is transverse to the nodal locus, and we may use $(\tau,s)$ as parameters. In these parameters the residue calculation gives exactly~\eqref{sew:period-unit}, with $c(\tau,s)=\widetilde c(q(\tau,s),s)$. Since $c(0,s)\ne0$, the inverse $\tau=\tau(t,s)$ is holomorphic.

The functions in~\eqref{sew:scaled-spectral-functions} satisfy $Y\,dX=\lambda_{\mathrm{odd}}$, and hence
\begin{equation}\label{sew:numerator}
 \lambda_{\mathrm{odd}}-\lambda_G
 =F\frac{dU}{U},\qquad
 F=\left[\frac{M((U+V)/(2\sqrt c),\tau,s)}{4c}-\frac12\right](U-V)^2.
\end{equation}
The relation $UV=t$ and symmetry under $U\leftrightarrow V$ give the unique convergent expansion $F=a_0+\sum_{m\ge1}a_m(U^m+V^m)$, with coefficients holomorphic in $(t,s)$. Both $\lambda_{\mathrm{odd}}$ and $\lambda_G$ have integral $2\pi i t$ around $\gamma$: for $\lambda_G$, substitute $V=t/U$ and read off the term $-t\,dU/U$. Since $\gamma$ is clockwise,
\[
 0=\frac1{2\pi i}\int_\gamma(\lambda_{\mathrm{odd}}-\lambda_G)=-a_0.
\]
At fixed $(t,s)$, the identity
\[
 d\!\left(\frac{U^m-V^m}{m}\right)
 =(U^m+V^m)\frac{dU}{U}
\]
then gives~\eqref{sew:source-primitive} by termwise integration.

At $t=0$, restrict to the branch $V=0$. Since $c(0,s)=M(0,0,s)/2$, the bracket in~\eqref{sew:numerator} vanishes at $U=0$, so $F(U,0)=O(U^3)$. Comparing coefficients gives $a_1(0,s)=a_2(0,s)=0$, and holomorphy makes both $a_1$ and $a_2$ divisible by $t$. This proves~\eqref{sew:low-modes}: for $m=1,2$ a nonzero coefficient requires $r\ge1$, while for $m\ge3$ the inequality $m+2r\ge3$ holds for every $r\ge0$.
\end{proof}

Subtracting $y_{\mathrm{ev}}$ from $y$ on the neck leaves the recursion kernel unchanged because $y_{\mathrm{ev}}$ has the same value on the two sheets over $x$. The stable correlators are therefore unchanged. The subtraction also preserves the residues defining $F_g$ in~\eqref{eq:closed-convention}. Indeed, at either colliding ramification point $a$, the linear loop equation for EO recursion states that $\omega_{g,1}+\sigma_a^*\omega_{g,1}$ is holomorphic. Since $H_{\mathrm{ev}}$ is holomorphic and fixed by $\sigma_a$, invariance of residues under $\sigma_a$ gives
\[
 \operatorname{Res}_{p=a}H_{\mathrm{ev}}\,\omega_{g,1}
 =\frac12\operatorname{Res}_{p=a}H_{\mathrm{ev}}
       \bigl(\omega_{g,1}+\sigma_a^*\omega_{g,1}\bigr)=0,
 \qquad g\ge2.
\]

\subsubsection{The normalized bidifferential}
\label{subsubsec:normalized-bidifferential}

To compute $B$, we use the capped-family construction of \cite[Theorem~3.1]{CCFZ}, retaining the transverse variation of the complement. Replacing the neck by two discs gives smooth compact curves. Their normalized bidifferential is the first term in the sewing expansion on the complement.

Fix $0<\rho<R$ within the chart $UV=t$ and restrict to $|t|<\rho^2$. Let $N_{t,s}=\{|U|\leq\rho,\ |V|\leq\rho\}$ be the neck and $C_{t,s}=\overline\Sigma_{t,s}\setminus N_{t,s}$ its complement. For $t\ne0$, the neck is the region $|t|/\rho\leq|U|\leq\rho$. Define $\overline\Sigma^{\mathrm{cap}}_{t,s}$ to be the compact curve obtained by attaching a separate disc $|h_\pm|<R$ at each of the two ends of $C_{t,s}$. Glue the discs along $\rho<|h_\pm|<R$ by
\[
 (U,V)=(h_+,t/h_+),\qquad (U,V)=(t/h_-,h_-).
\]
Thus $h_+$ extends $U$ across one disc and $h_-$ extends $V$ across the other. The disc centers $h_\pm=0$ are two new smooth points. The transition maps are holomorphic in the fiber coordinates and $(t,s)$, so the capped curves form a smooth proper holomorphic family, retaining the given variation of the complement $C_{t,s}$. At $t=0$, these centers correspond to $p_\pm$, and $\overline\Sigma^{\mathrm{cap}}_{0,s}$ is the normalization of $\overline\Sigma_{0,s}$. Sewing the discs by $h_+h_-=t$ recovers the original fiber.

Let $B^{\mathrm{cap}}$ be the fundamental bidifferential on $\overline\Sigma^{\mathrm{cap}}_{t,s}$, normalized by the inherited $A$-cycles on each component and set to zero between distinct components. Its central value is $B^{\mathrm{cap}}(0,s)=B_s^{\mathrm{back}}$.

The complement $C_{t,s}$ is naturally an open subset of both the original and capped curves. We use this identification to restrict $B^{\mathrm{cap}}$ to $C_{t,s}\times C_{t,s}$.

For $a,b\in\{+,-\}$ and $m,n\ge1$, define
\begin{equation}\label{sew:seam-matrix}
 \begin{aligned}
 b_{am}(p)&=[h_a^{m-1}dh_a]B^{\mathrm{cap}}(p,h_a),\\
 E_{am,bn}&=[h_a^{m-1}h_b^{n-1}dh_a\,dh_b]
                 B^{\mathrm{cap,reg}}(h_a,h_b),\\
 (P_t)_{am,bn}&=-\delta_{b,\bar a}\delta_{mn}\frac{t^m}{m},
 \qquad \Theta_{+,m}=U^{m-1}dU,\quad \Theta_{-,m}=V^{m-1}dV.
 \end{aligned}
\end{equation}
Here $\bar a$ denotes the opposite cap. When both arguments lie in cap $a$, set
\[
 B^{\mathrm{cap,reg}}(p,q)=B^{\mathrm{cap}}(p,q)
 -\frac{dh_a(p)\,dh_a(q)}{(h_a(p)-h_a(q))^2}.
\]
Between distinct caps, $B^{\mathrm{cap,reg}}$ equals $B^{\mathrm{cap}}$, which is zero between distinct capped components.

Each $b_{am}$ is a meromorphic one-form on the capped curve, with its only pole at the center of cap $a$; in the formulas below, we use its restriction to $C_{t,s}$. The forms $\Theta_{+,m}$ and $\Theta_{-,m}$ are defined on the neck $N_{t,s}$ in the coordinates $U$ and $V$. Regard $b$ and $\Theta$ as row vectors. We identify $N_{t,s}$ with the closed annulus $|t|/\rho\leq|w|\leq\rho$ in $G_t$ by $w=U$. Thus $\lambda_G$ is the one-form of~\eqref{gau:curve}, and on $N_{t,s}\times N_{t,s}$ we use $B_G(p,q)=dU(p)\,dU(q)/(U(p)-U(q))^2$.

For $t\ne0$, suppress the parameters and write $N=N_{t,s}$ and $C=C_{t,s}$. Define the restrictions of $B$ by
\[
 B_{CC}=B|_{C\times C},\qquad
 B_{NC}=B|_{N\times C},\qquad
 B_{NN}=B|_{N\times N}.
\]

\begin{proposition}\label{sew:plumbing}
The restrictions of $B$, normalized by the adapted $A$-cycles, satisfy
\begin{equation}\label{sew:plumbing-formulas}
 \begin{aligned}
 B_{CC}(p,q)&=B^{\mathrm{cap}}(p,q)+b(p)P_t(I-EP_t)^{-1}b(q)^{\mathsf T},\\
 B_{NC}(p,q)&=\Theta(p)(I-EP_t)^{-1}b(q)^{\mathsf T},\\
 B_{NN}(p,q)&=B_G(p,q)+\Theta(p)(I-EP_t)^{-1}E\Theta(q)^{\mathsf T}.
 \end{aligned}
\end{equation}
The matrix series converge normally for sufficiently small $t$, locally uniformly in $s$ and on smaller coordinate domains. The entries of $(I-EP_t)^{-1}$ and $(I-EP_t)^{-1}E$, and the one-forms $b_{am}$, are holomorphic in $(t,s)$. In particular,
\begin{equation}\label{sew:kernel-orders}
 B_{CC}-B^{\mathrm{back}}=O(t).
\end{equation}
This estimate is understood after pullback to fixed coordinate discs in the complement. Each neck variable in $B_{NC}$ and $B_{NN}-B_G$ is expanded in the one-forms $\Theta_{a,m}$ with $m\ge1$; in particular, there is no $dU/U$ term.
\end{proposition}

\begin{proof}
The one-node expansions are due to Yamada \cite[Theorems~4 and~6]{Yamada}; the convergent construction follows Grushevsky--Krichever--Norton \cite{GKN} and Hu--Norton \cite[Section~3]{HuNorton}, with the varying capped family treated in Chen--Chen--Fang--Zong \cite[Theorem~3.1]{CCFZ}. We write the construction in the coefficient notation above.

Fix $q$ in the complement $C$. On the capped curve, $B^{\mathrm{cap}}(p,q)$ already has the required double pole at $p=q$ and zero inherited $A$-periods. To obtain $B(p,q)$, we must make the expansions at the two caps agree under $h_+h_-=t$. Since the cap centers are removed when we sew, we may add differentials with poles there.

A regular term at one cap becomes a polar term at the other:
\[
 h_{\bar a}^{m-1}dh_{\bar a}=-t^m h_a^{-m-1}dh_a.
\]
The differential $b_{am}/m$ has principal part $h_a^{-m-1}dh_a$ and zero inherited $A$-periods. Thus the regular term on the left requires the correction $-t^m b_{am}/m$ on the capped curve. The matrix $P_t$ records these corrections. Each correction also changes the regular terms at both caps; those changes are recorded by $E$.

For fixed $q\in C$, let $f(q)=(f_{am}(q))$ be the column whose entry $f_{am}(q)$ is the coefficient of $h_a^{m-1}dh_a(p)$ in the regular part of the corrected differential in $p$ at the center of cap $a$. The total correction on the complement must then be $b(p)P_t f(q)$. The initial regular coefficients are $b(q)^{\mathsf T}$, and the correction contributes $EP_t f(q)$. Therefore
\[
 f(q)=b(q)^{\mathsf T}+EP_t f(q),
 \qquad f(q)=(I-EP_t)^{-1}b(q)^{\mathsf T}.
\]
On the neck, the regular terms from the two ends together give $\Theta(p)f(q)$. We have therefore obtained the first two expressions in~\eqref{sew:plumbing-formulas}.

The inverse means the series $I+EP_t+(EP_t)^2+\cdots$: each term adds the regular coefficients produced by one more correction. Bounds on the Taylor coefficients on smaller cap discs show that this series and the resulting sums over $(a,m)$ converge normally for small $t$, with derivatives on smaller domains. These estimates are uniform in $s$. The two expressions consequently define a differential on the sewn curve. It has the prescribed double pole, zero inherited $A$-periods, and zero neck period, the last because its Laurent expansion has no $dU/U$ term. These are the periods of the adapted $A$-polarization, so uniqueness identifies it with $B(p,q)$.

To obtain the third formula, continue the expression for $B_{NC}(p,q)$ in its second variable from the annulus $\rho<|U(q)|<R$ at the positive end into the neck. Put $\mathsf R=(I-EP_t)^{-1}$ and expand in $U(q)$. The regular part of $b(q)^{\mathsf T}$ gives the terms $\Theta_{+,n}(q)$, with coefficients $\Theta(p)\mathsf R E$. For the principal parts, use $\mathsf R=I+\mathsf R EP_t$. The identity term gives the Gaussian bidifferential $B_G$: for $|U(p)|<|U(q)|$,
\[
 \sum_{m\ge1}mU(p)^{m-1}U(q)^{-m-1}dU(p)\,dU(q)
 =\frac{dU(p)\,dU(q)}{(U(p)-U(q))^2}=B_G(p,q).
\]
In the remaining term, $P_t$ changes each principal part $mU(q)^{-m-1}dU(q)$ into $V(q)^{m-1}dV(q)$. This supplies the terms $\Theta_{-,n}(q)$ with the same coefficients $\Theta(p)\mathsf R E$. Combining the two ends gives
\[
 B_{NN}(p,q)=B_G(p,q)+\Theta(p)\mathsf R E\Theta(q)^{\mathsf T}.
\]
The identity extends from this annulus to the neck by analytic continuation.

Finally, the capped curves form a smooth proper family with fixed inherited $A$-cycles, so $B^{\mathrm{cap}}$ and its cap coefficients are holomorphic in $(t,s)$. Normal convergence gives the same conclusion for the coefficients in all three formulas. Since $B^{\mathrm{cap}}(0,s)=B_s^{\mathrm{back}}$ and $P_t=O(t)$, the correction on the complement is $O(t)$ after pullback to fixed coordinate discs there. The neck formulas involve only the one-forms $\Theta_{a,m}$ with $m\ge1$. This proves~\eqref{sew:kernel-orders}.
\end{proof}

To compare the spectral data as $t$ varies, choose pairwise disjoint coordinate discs $D_\alpha$ around the background ramification points $a_\alpha$, outside the cap discs $|h_\pm|\leq\rho$. Choose a Morse coordinate $\xi_\alpha$ with $x-x(a_\alpha)=\xi_\alpha^2/2$ on each disc, and fix its sign. Extend these coordinates holomorphically with $(t,s)$ around the corresponding ramification sections in the complement. Points with the same value of $\xi_\alpha$ are then identified across fibers. Expressed in these coordinates, $\lambda$ and $B$ depend holomorphically on $(t,s)$, and their differences from the corresponding background forms are $O(t)$.

On the neck, use the spectral functions $X,Y$ of~\eqref{sew:scaled-spectral-functions}. To keep the two neck ramification points at fixed locations, pass to the cover $t=\epsilon^2$ and write $U=\epsilon\zeta$, $V=\epsilon/\zeta$. These points are then at $\zeta=\pm1$. Since both $U$ and $V$ have a factor of $\epsilon$, each one-form $\Theta_{\pm,m}$ has a factor of $\epsilon^m=t^{m/2}$.

\subsubsection{Contour integrals}
\label{subsubsec:contour-integrals}

Fix $t\ne0$ and use the neck $N=N_{t,s}$ defined above, whose interior contains the two colliding ramification points. Let $\phi$ be a holomorphic function on a neighborhood of $N$, so that its Laurent expansion in $U$ converges on $N$. For a meromorphic one-form $\omega$ with no poles on the boundary of $N$, define
\begin{equation}\label{sew:total-cycle}
 \mathcal C_\phi\omega
 =\frac{1}{2\pi i}
   \left(\int_{\mathrm{outer}}-\int_{\mathrm{inner}}\right)
     \phi\omega.
\end{equation}
The outer and inner circles are $|U|=\rho$ and $|U|=|t|/\rho$, respectively, both taken counterclockwise in the $U$-coordinate. Their difference is the positively oriented boundary of the neck. We use the following functions and their differentials:
\begin{equation}\label{sew:neck-modes}
 \begin{aligned}
 \phi_{+,m}&=\frac{U^m}{m},
 & e_{+,m}=d\phi_{+,m}&=U^{m-1}dU,\\
 \phi_{-,m}&=-\frac{V^m}{m},
 & e_{-,m}=d\phi_{-,m}&=-V^{m-1}dV,
 \qquad m\ge1.
 \end{aligned}
\end{equation}

On each background disc $D_\alpha$ with its chosen Morse coordinate $\xi_\alpha$, set
\[
 \phi_{\alpha,m}=\frac{\xi_\alpha^m}{m},
 \qquad e_{\alpha,m}=d\phi_{\alpha,m}
       =\xi_\alpha^{m-1}d\xi_\alpha,\qquad m\ge1.
\]
For a holomorphic function $\phi$ on a neighborhood of $\overline{D_\alpha}$ and a meromorphic one-form $\omega$ with no poles on $\partial D_\alpha$, define
\[
 \mathcal C_\phi\omega
 =\frac{1}{2\pi i}\int_{\partial D_\alpha}\phi\omega,
\]
where the boundary is oriented counterclockwise in $\xi_\alpha$.

The index $I$ specifies a neck end $+$ or $-$, or a background ramification point $\alpha$, together with an integer $m\ge1$. Write $\mathcal C_I=\mathcal C_{\phi_I}$ and $e_I=d\phi_I$. Thus $\phi_I$ is the chosen primitive of $e_I$. We regard $\phi_I$ and $e_I$ as a function and a one-form on the disjoint union $N\sqcup\bigsqcup_\alpha D_\alpha$. The functions $\phi_{\pm,m}$ and one-forms $e_{\pm,m}$ are defined on all of $N$ and set to zero on every $D_\alpha$. The functions $\phi_{\alpha,m}$ and one-forms $e_{\alpha,m}$ are defined on $D_\alpha$ and set to zero on $N$ and on every $D_\beta$ with $\beta\ne\alpha$.

On this disjoint union, define the reference bidifferential by
\[
 B_{\mathrm{ref}}(p,q)=
 \begin{cases}
 B_G(p,q),&p,q\in N,\\
 B_s^{\mathrm{back}}(p,q),&p\in D_\alpha,\ q\in D_\beta,\\
 0,&\text{one point lies in }N\text{ and the other in some }D_\alpha.
 \end{cases}
\]
The second line includes pairs of distinct background discs.

\begin{lemma}\label{sew:cycles}
Integration against $\phi_I$ turns the reference bidifferential into the one-form $e_I=d\phi_I$, and gives zero on every $e_J$. More precisely, for all indices $I,J$ and every point $p$ in the interior of $N$ or in one of the discs $D_\alpha$,
\begin{equation}\label{sew:isotropy}
 \mathcal C_I^{(z)}B_{\mathrm{ref}}(z,p)=e_I(p),
 \qquad \mathcal C_I e_J=0.
\end{equation}
In the first integral, $z$ is the integration variable and $p$ is held fixed.

On a stable Gaussian differential $\omega=\omega^{G_t}_{g,n}$, with $n\ge1$ and $2g-2+n>0$, the neck integrals equal the residues at $0$ and $\infty$ defined in~\eqref{gau:end-moments}:
\[
 \mathcal C_{\phi_{\pm,m}}\omega=\mathcal L_m^\pm\omega,
 \qquad m\ge1.
\]
Each operation acts on one variable of $\omega$, with the other variables held fixed.
\end{lemma}

\begin{proof}
Fix $I$ and $p$, and consider the contour integral of $\phi_I(z)B_{\mathrm{ref}}(z,p)$. For $I=(\pm,m)$, we integrate over the boundary of the neck $N$; for $I=(\alpha,m)$, we integrate over $\partial D_\alpha$. Suppose that $p$ lies in this neck or disc. The integrand has no poles there except possibly at $z=p$. In a local coordinate $u$, the singular part of the kernel is $du(z)\,du(p)/(u(z)-u(p))^2$. Its double pole turns the linear term of the Taylor expansion of $\phi_I(z)$ into the residue $d\phi_I(p)$. Thus the residue theorem gives
\[
 \mathcal C_I^{(z)}B_{\mathrm{ref}}(z,p)
 =\operatorname{Res}_{z=p}\phi_I(z)B_{\mathrm{ref}}(z,p)
 =d\phi_I(p)=e_I(p).
\]
If $p$ lies elsewhere, $e_I(p)=0$ by definition, and the integral also vanishes. Indeed, $B_{\mathrm{ref}}$ is zero between the neck and a background disc. Between two distinct background discs, it is holomorphic in $z$, so its weighted boundary integral vanishes.

Next, the one-form $\phi_Ie_J$ is holomorphic on the whole neck or disc, so its integral over the total oriented boundary vanishes. Hence $\mathcal C_Ie_J=0$. On the neck, the outer and inner circle integrals have the same value when both circles are oriented counterclockwise. Their difference in~\eqref{sew:total-cycle} is consequently zero.

Finally, take a stable Gaussian differential and hold all but one of its variables fixed. Its only possible poles in the remaining variable $U$ are the ramification points $U=\pm\sqrt t$, both inside the neck. We may therefore shrink the inner circle toward $0$ and expand the outer circle toward $\infty$ without crossing a pole of $\omega$.

With the weight $\phi_{+,m}=U^m/m$, the integrand is regular at $U=0$, so the inner integral vanishes. The outer circle contributes minus the residue at infinity: a counterclockwise circle in $U$ is clockwise in the local coordinate $1/U$ at infinity. Hence
\[
 \mathcal C_{\phi_{+,m}}\omega
 =-\operatorname{Res}_{U=\infty}\frac{U^m}{m}\omega
 =\mathcal L_m^+\omega.
\]
With the weight $\phi_{-,m}=-V^m/m=-t^mU^{-m}/m$, the integrand is regular at infinity, so the outer integral vanishes. The definition of $\mathcal C_\phi$ subtracts the inner integral. This minus sign cancels the one in $\phi_{-,m}$, giving
\[
 \mathcal C_{\phi_{-,m}}\omega
 =\operatorname{Res}_{U=0}\frac{V^m}{m}\omega
 =\mathcal L_m^-\omega.
\]
\end{proof}

When several integration variables lie in the same neck or coordinate disc, choose nested contours avoiding the diagonal poles: the contour for a later integration lies in the interior of the domain used for the earlier integration.

Let $\mathbf v_I$ be a formal basis vector for each background index $I=(\alpha,m)$ and each neck index $I=(\pm,m)$. For $n\ge1$ and $2h-2+n>0$, define the background tensor first on background indices $I_j=(\alpha_j,m_j)$ by
\begin{equation}\label{sew:background-tensor}
 \widehat{\mathcal B}_{h,n}(I_1,\ldots,I_n;s)
 =\left(\prod_{j=1}^n\mathcal C_{I_j}\right)
   \omega^{\mathrm{back}}_{h,n}.
\end{equation}
Here $\mathcal C_{I_j}$ integrates the $j$th variable around $\partial D_{\alpha_j}$. Set the component to zero if any $I_j$ is a neck index. These components determine $\widehat{\mathcal B}_{h,n}$ as a multilinear form in the basis vectors $\mathbf v_I$.

For neck indices $I_j=(a_j,m_j)$, with $a_j\in\{+,-\}$, apply the contour operators to the Gaussian correlator $\omega^{G_t}_{h,n}$. By Lemma~\ref{sew:cycles} and the equality of the Gaussian residue moments at the two ends, the result depends only on the integers $m_j$. We use the same symbol $\mathcal A_{h,n}$ for the pullback of the Gaussian tensor under $\mathbf v_{\pm,m}\mapsto\mathbf v_m$, so
\[
 \mathcal A_{h,n}(\mathbf v_{a_1,m_1},\ldots,\mathbf v_{a_n,m_n};t)
 =\mathcal A_{h,n}(m_1,\ldots,m_n;t).
\]
Extend it multilinearly, setting its value to zero whenever any argument is a background basis vector.

We now compare the spectral data on the neck and the background discs with their reference values. Use the Morse-coordinate identifications on the background discs and $\lambda=\lambda_{\mathrm{odd}}$ on the neck. Let $\lambda_{\mathrm{ref}}$ equal $\lambda_G$ on $N$ and the spectral one-form of $\mathcal B_s$ on each $D_\alpha$. Set $\delta\lambda=\lambda-\lambda_{\mathrm{ref}}$ and $\delta B=B-B_{\mathrm{ref}}$. Expand these differences in the one-forms $e_I$:
\begin{equation}\label{sew:couplings}
 \delta\lambda=\sum_I d_I(t,s)e_I,\qquad
 \delta B=\sum_{I,J}k_{IJ}(t,s)e_I\otimes e_J,
 \qquad k_{IJ}=k_{JI}.
\end{equation}
We call $d_I$ the source coefficients and $k_{IJ}$ the edge coefficients. Both sums include background and neck indices. The coefficients are obtained by expanding the formulas in Propositions~\ref{sew:local-form} and~\ref{sew:plumbing}, using the Morse-coordinate identifications on the background discs; they are holomorphic in $(t,s)$. The one-form difference is $O(t)$ on each background disc, and the bidifferential difference is $O(t)$ when both variables lie on background discs. Hence
\[
 d_{\alpha,m}=O(t),\qquad
 k_{(\alpha,m),(\beta,n)}=O(t).
\]

On the neck, $\delta\lambda=\lambda_{\mathrm{odd}}-\lambda_G$, and~\eqref{sew:source-primitive} gives its primitive $\Psi_A$. Differentiating that formula gives
\[
 \delta\lambda=d\Psi_A
 =\sum_{m\ge1}a_m\,d\!\left(\frac{U^m-V^m}{m}\right)
 =\sum_{m\ge1}a_m(e_{+,m}+e_{-,m}).
\]
Thus $e_{+,m}$ and $e_{-,m}$ have the same coefficient in~\eqref{sew:couplings}: $d_{+,m}=d_{-,m}=a_m$.

To find $k_{IJ}$, subtract $B_{\mathrm{ref}}$ from the bidifferential in~\eqref{sew:plumbing-formulas} and express the result in the one-forms $e_I$. For the neck factors, substitute $\Theta_{+,m}=e_{+,m}$ and $\Theta_{-,m}=-e_{-,m}$. For example,
\[
 \begin{aligned}
 \Theta_{+,m}\otimes\Theta_{-,n}&=-e_{+,m}\otimes e_{-,n},\\
 \Theta_{-,m}\otimes\Theta_{-,n}&=e_{-,m}\otimes e_{-,n}.
 \end{aligned}
\]
These signs are included in the coefficients $k_{IJ}$.

The primitive series
\begin{equation}\label{sew:coupling-primitives}
 \Psi(p)=\sum_I d_I\phi_I(p),\qquad
 \mathcal K(p,q)=\sum_{I,J}k_{IJ}\phi_I(p)\phi_J(q)
\end{equation}
converge normally on smaller closed neck annuli and background discs, and on products of these domains. The convergence is locally uniform in $(t,s)$ for $t\ne0$, with derivatives on smaller domains, and holds on the chosen integration contours. Indeed, the expansions of $\delta\lambda$ and $\delta B$ in~\eqref{sew:couplings} are convergent Taylor and Laurent series. Their neck expansions have no $dU/U$ term, so termwise integration divides each coefficient by its positive index and gives~\eqref{sew:coupling-primitives}. The resulting holomorphic functions satisfy $d\Psi=\delta\lambda$ and $d_p d_q\mathcal K=\delta B$.

\begin{unnumberedremark}
On the correlators and bidifferentials used here, the contour integrals $\mathcal C_I$ are unchanged when the neck cutting radius $\rho$ and the background disc radii vary through deformations avoiding poles, including diagonal poles in repeated integrations. The coefficients $d_I,k_{IJ}$ are the Taylor and Laurent coefficients of $\delta\lambda$ and $\delta B$ in the fixed one-forms $e_I$. Their values are independent of $\rho$, the cap radius $R$, and the disc radii, and are preserved on smaller coordinate domains.
\end{unnumberedremark}

\section{Graph expansion and factorization}
\label{sec:proof}
\subsection{Relative variation and the free energy}
\label{sec:variation}

Fix $(t,s)$ with $t\ne0$. On each background disc, use the same Morse coordinate $\xi_\alpha$ for the original and reference data. Their spectral $x$-maps then differ only by the additive branch values $x(a)$, which do not enter the recursion. On the neck, both use the spectral $x$-map $U+V$. We can therefore keep these local $x$-maps fixed while varying the one-form and bidifferential:
\[
 \lambda_{z_\lambda}=\lambda_{\mathrm{ref}}+z_\lambda\delta\lambda,
 \qquad B_{z_B}=B_{\mathrm{ref}}+z_B\delta B.
\]
The parameters $z_\lambda,z_B$ are independent. At $(0,0)$ these are the separated background and Gaussian neck data; at $(1,1)$ they are the original local data, with $\lambda=\lambda_{\mathrm{odd}}$ on the neck. Write $\omega_{g,n}(z_\lambda,z_B)$ and $F_g(z_\lambda,z_B)$ for their correlators and free energies. For fixed $t\ne0$, these are analytic near $(0,0)$, and we first work with their formal Taylor series there. On the diagonal $z_\lambda=z_B=z$, abbreviate the data and invariants to $(\lambda_z,B_z)$, $\omega_{g,n}(z)$ and $F_g(z)$.

The primitives $\Psi$ and $\mathcal K$ in~\eqref{sew:coupling-primitives} converge normally on smaller contour domains. For a fixed graph, summing the source and edge indices therefore replaces their factors by $-\Psi$ and $\mathcal K$ inside the multiple contour integral of the vertex correlators. Choose nested contours for variables in the same neck or coordinate disc, avoiding diagonal poles. On these contours the sums converge absolutely, locally uniformly in $(t,s)$ for $t\ne0$, so they may be interchanged with the contour integrations. At any fixed genus, number of ordinary leaves, and powers of $z_\lambda,z_B$, stability bounds the number of vertices and leaves. Only finitely many underlying graphs then occur. These facts give analytic coefficients to the formal graph series used below.

The operators $\mathsf A$ and $\mathsf Q$ integrate one or two additional variables using the contour operators $\mathcal C_I$. More explicitly, for a differential $\eta$ in the displayed variables, set
\[
 \begin{aligned}
 (\mathsf A\eta)(p_1,\ldots,p_n)
 &=\sum_I d_I\mathcal C_I^{(w)}
       \eta(w,p_1,\ldots,p_n),\\
 (\mathsf Q\eta)(p_1,\ldots,p_n)
 &=\sum_{I,J}k_{IJ}\mathcal C_I^{(w)}\mathcal C_J^{(w')}
       \eta(w,w',p_1,\ldots,p_n).
 \end{aligned}
\]
The superscripts specify the integration variables; the points $p_1,\ldots,p_n$ are held fixed. The sums include both background and neck indices, with coefficients $d_I$ and $k_{IJ}$ from~\eqref{sew:couplings}.

Varying $B$ also changes the dependence on each variable already present in a correlator. To describe this contribution, use the chosen primitives $\phi_I$ to set
\[
 \kappa(w,p)=\sum_{I,J}k_{IJ}\phi_I(w)e_J(p),
 \qquad d_w\kappa(w,p)=\delta B(w,p).
\]
Thus $\kappa$ is a function of $w$ and a one-form in $p$. For a meromorphic one-form $\omega$, define
\[
 (\mathsf L\omega)(p)=\sum_a\operatorname{Res}_{w=a}
       \omega(w)\kappa(w,p).
\]
The sum runs over all background ramification points and the two ramification points in the neck. For a correlator with several variables, $\mathsf L_i$ applies this operation to its $i$th variable, keeping the others fixed.

\begin{proposition}[Variational formula]\label{sew:variation}
For $n\ge1$ and $2g-2+n>0$, write $\mathbf p=(p_1,\ldots,p_n)$ and $\mathbf p_P=(p_j)_{j\in P}$. Then
\begin{equation}\label{sew:marked-variation}
 \begin{aligned}
 \partial_{z_\lambda}\omega_{g,n}
 ={}&-\mathsf A\omega_{g,n+1},\\
 \partial_{z_B}\omega_{g,n}
 ={}&\frac12\mathsf Q\omega_{g-1,n+2}(w,w',\mathbf p)
      +\sum_{i=1}^n\mathsf L_i\omega_{g,n}\\
 &+\frac12\mathsf Q\left(
      \sum_{\substack{h+h'=g,\ P\sqcup P'=\{1,\ldots,n\}\\
                 2h-1+|P|>0,\ 2h'-1+|P'|>0}}
         \omega_{h,|P|+1}(w,\mathbf p_P)
         \omega_{h',|P'|+1}(w',\mathbf p_{P'})\right).
 \end{aligned}
\end{equation}
Both identities hold at $(z_\lambda,z_B)$, and $\mathsf Q$ acts on $w,w'$. Negative-genus terms are zero. On the diagonal, $\partial_z\omega_{g,n}(z)$ is the sum of the two right-hand sides.
\end{proposition}

These are the fixed-$x$ formulas of Kazarian, proved by Alexandrov--Bychkov--Dunin-Barkowski--Kazarian--Shadrin in~\cite[Theorem~A.1, equation~(33) in its proof]{ABDKS}, written in the convention $\omega_{0,1}=-\lambda$. They apply to local spectral data. In the recursion graph expansion of Eynard--Orantin~\cite[Section~4.5]{EO2007}, differentiating the kernel denominator gives the negative source term. Differentiating an internal kernel cuts an edge and gives the connected and disconnected terms; differentiating the dependence on an existing variable gives $\mathsf L_i$.

\begin{proposition}\label{sew:vacuum-variation}
For $g\ge2$, the free energy satisfies
\begin{equation}\label{sew:closed-variation}
 \begin{aligned}
 \partial_{z_\lambda}F_g&=-\mathsf A\omega_{g,1},\\
 \partial_{z_B}F_g&=\frac12\mathsf Q\left(
       \omega_{g-1,2}(w,w')
       +\sum_{h=1}^{g-1}\omega_{h,1}(w)\omega_{g-h,1}(w')\right).
 \end{aligned}
\end{equation}
\end{proposition}

\begin{proof}
Differentiate the residue formula~\eqref{eq:closed-convention} with its varying primitive:
\[
 F_g=\frac{1}{2g-2}\sum_a\operatorname{Res}_a\Phi_a\omega_{g,1},
 \qquad d\Phi_a=\lambda_{z_\lambda}\text{ near }a.
\]
Near each $a$, $\partial_{z_\lambda}\Phi_a=\Psi+c_a$ for a constant $c_a$. Since $\omega_{g,1}$ has zero residue at $a$, differentiating the primitives contributes $\mathsf A\omega_{g,1}/(2g-2)$. By the dilaton equation~\cite[Theorem~4.7]{EO2007}, applying the residue formula to the source term in~\eqref{sew:marked-variation} contributes $-(2g-1)\mathsf A\omega_{g,1}/(2g-2)$. Their sum is the first identity in~\eqref{sew:closed-variation}. For the $z_B$-derivative, the connected edge term has dilaton factor $2(g-1)-2+2=2g-2$. In a disconnected term, the variable integrated against $\Phi_a$ can lie on either factor; their dilaton factors sum to $(2h-1)+(2(g-h)-1)=2g-2$. These factors cancel the denominator. The term $\mathsf L_1\omega_{g,1}$ is holomorphic in its output variable, so its product with $\Phi_a$ has zero residue. This proves the second identity.
\end{proof}

\subsection{Integration as a stable graph sum}
\label{sec:graphs}

Equations~\eqref{sew:marked-variation} and~\eqref{sew:closed-variation} determine the Taylor expansion from the separated background and Gaussian data to $(\lambda_z,B_z)$. We first record the cancellation used when integrating the one-form variation.

\begin{lemma}\label{sew:source-cancellation}
Let $\omega_{0,n}(z)$ be the genus-zero correlators for $\lambda_z=\lambda_{\mathrm{ref}}+z\delta\lambda$ with $B=B_{\mathrm{ref}}$ fixed. Starting from a stable correlator, integrating $\ell$ variables with $\mathsf A$ gives zero whenever at most two variables remain. More precisely,
\begin{equation}\label{sew:effective-unstable}
 \mathsf A^\ell\omega_{0,r+\ell}(z)=0,
 \qquad r=0,1,2,\qquad \ell\ge3-r.
\end{equation}
Here each application of $\mathsf A$ acts on a different variable, and $r$ is the number of variables left unintegrated. The identity is a sum over the source indices.
\end{lemma}

\begin{proof}
The function $\Psi=\sum_I d_I\phi_I$ is a holomorphic primitive of $\delta\lambda$ on the neck and on each background disc. At a ramification point $a$, choose a Morse coordinate $\xi$ with $x-x(a)=\xi^2/2$ and write
\[
 \delta\lambda=\delta y(\xi)\xi\,d\xi.
\]
The function $\delta y$ is holomorphic. On the neck, this follows from the factor $(U-V)^2$ in~\eqref{sew:numerator}, since $dx=(U-V)dU/U$. Thus $\Psi=\Psi(a)+O(\xi^2)$ has no linear term in $\xi$.

With the other two variables fixed, the polar part of $\omega_{0,3}(z)$ at $a$ is a multiple of $d\xi/\xi^2$ by~\cite[Theorem~4.1]{EO2007}. Multiplying by $\Psi$ therefore gives zero residue. The contour definition of $\mathsf A$ sums these residues over all ramification points, so $\mathsf A\omega_{0,3}(z)=0$. Applying $\mathsf A$ to further variables gives
\[
 \mathsf A^{3-r}\omega_{0,3}(z)=0,
 \qquad r=0,1,2.
\]

Since $B$ is fixed, Proposition~\ref{sew:variation} reduces to $\partial_z\omega_{0,n}(z)=-\mathsf A\omega_{0,n+1}(z)$ for $n\ge3$. The operator $\mathsf A$ is independent of $z$, so each differentiation adds one variable and one application of $-\mathsf A$. Differentiating the preceding identity $j$ times gives
\[
 0=\partial_z^j\bigl(\mathsf A^{3-r}\omega_{0,3}(z)\bigr)
   =(-1)^j\mathsf A^{j+3-r}\omega_{0,j+3}(z).
\]
Taking $j=\ell+r-3\ge0$ proves~\eqref{sew:effective-unstable}.
\end{proof}

\begin{definition}[Decorated graphs]\label{sew:decorated-graphs}
Fix a finite set $\mathscr R$ and two disjoint index sets
\[
 \mathfrak I_A=\{+,-\}\times\mathbb Z_{>0},
 \qquad \mathfrak I_B=\mathscr R\times\mathbb Z_{>0}.
\]
A connected decorated graph $\Gamma$ consists of the following data:
\begin{itemize}
\item \textbf{Vertices.} A finite nonempty set $V(\Gamma)$. Each vertex $v$ has a genus $h_v\in\mathbb Z_{\ge0}$ and a type $c_v\in\{A,B\}$.
\item \textbf{Half-edges.} A finite set $H(\Gamma)$ of half-edges, also called flags, each incident to one vertex. A half-edge $f$ at $v$ has an index $I_f\in\mathfrak I_{c_v}$. Write $H(v)$ for the incident half-edges and $n_v=|H(v)|$ for the valence.
\item \textbf{Edges.} A set $E(\Gamma)$ of disjoint unordered pairs of half-edges. Each pair forms an edge between the incident vertices. Loops and multiple edges are allowed, and the edges connect all vertices.
\item \textbf{Leaves.} The unpaired half-edges are divided into ordinary leaves $L^o(\Gamma)=\{l_1,\ldots,l_n\}$, numbered from $1$ to $n$, and unnumbered source leaves $S(\Gamma)$.
\end{itemize}
The graph is stable if, at every vertex,
\[
 2h_v-2+n_v>0.
\]
Here $n_v\ge0$ counts edge ends and both kinds of leaves. The genus of $\Gamma$ is
\[
 g(\Gamma)=\sum_vh_v+|E(\Gamma)|-|V(\Gamma)|+1.
\]
Isomorphisms preserve incidence, edge pairings, both kinds of leaves, ordinary-leaf numbers, and all vertex and half-edge decorations. Let $\operatorname{Aut}\Gamma$ be the automorphism group. Write $\mathfrak G_{g,n}$ for the set of isomorphism classes of stable decorated graphs of genus $g$ with $n$ ordinary leaves, and set
\[
 \mathfrak G'_{g,n}
 =\{[\Gamma]\in\mathfrak G_{g,n}:n_v>0
       \text{ for every vertex }v\}.
\]
For graphs with no ordinary leaves, abbreviate
\[
 \mathfrak G_g=\mathfrak G_{g,0},
 \qquad \mathfrak G'_g=\mathfrak G'_{g,0}.
\]
\end{definition}

For the spectral data, take $\mathscr R$ to be the set of background ramification points. Types $A$ and $B$ correspond to the neck and background. Assign the following weights:
\begin{itemize}
\item \textbf{Ordinary leaves.} An ordinary leaf $l_j$ has weight $1$ and a prescribed index $I_{l_j}$. It records the corresponding contour integral of a correlator. In an $n$-point graph sum these $n$ indices are held fixed.
\item \textbf{Source leaves.} A source leaf $f$ has weight $-d_{I_f}(t,s)$, with $d_I$ defined in~\eqref{sew:couplings}.
\item \textbf{Edges.} An edge $\{f,f'\}$ has weight $k_{I_fI_{f'}}(t,s)$, with $k_{IJ}$ defined in~\eqref{sew:couplings}.
\item \textbf{Vertices.} A vertex $v$ with $n_v>0$ has weight
\[
 T_v=
 \begin{cases}
 \mathcal A_{h_v,n_v}\bigl((m_f)_{f\in H(v)};t\bigr),
       &c_v=A,\\[2pt]
 \widehat{\mathcal B}_{h_v,n_v}\bigl((I_f)_{f\in H(v)};s\bigr),
       &c_v=B.
 \end{cases}
\]
At a type-$A$ vertex, write $I_f=(a_f,m_f)$ with $a_f\in\{+,-\}$. Both tensors are symmetric, so no ordering of $H(v)$ is needed. If $n_v=0$, stability gives $h_v\ge2$, and we set
\[
 T_v=
 \begin{cases}
 F_{h_v}(G_t)=C_{h_v}t^{2-2h_v},&c_v=A,\\[2pt]
 F_{h_v}(\mathcal B_s),&c_v=B.
 \end{cases}
\]
\end{itemize}
Each flag $f\in H(v)$ supplies one argument of the vertex correlator, integrated by $\mathcal C_{I_f}$ in the definition of $T_v$. This includes source leaves and both ends of every edge. Thus the source and edge weights multiply tensor components in which these contour integrations have already been performed.

The weight of the decorated graph is the product of these factors:
\[
 w(\Gamma;t,s)=
 \prod_{v\in V(\Gamma)}T_v
 \prod_{f\in S(\Gamma)}(-d_{I_f})
 \prod_{\{f,f'\}\in E(\Gamma)}k_{I_fI_{f'}}.
\]
The vacuum graph sets $\mathfrak G_g$ and $\mathfrak G'_g$ have no ordinary leaves. Summing their decorated isomorphism classes includes all half-edge index assignments.

\begin{theorem}[The connected graph sum for the difference]\label{sew:graph}
For $g\ge2$, the difference of free energies is the following formal series in $z$:
\begin{equation}\label{sew:graph-sum}
 \begin{aligned}
 &F_g(\lambda_z,B_z)-F_g(\mathcal B_s)-F_g(G_t)\\
 &\hspace{1cm}=
 \sum_{[\Gamma]\in\mathfrak G'_g}
 \frac{z^{|S(\Gamma)|+|E(\Gamma)|}}{|\operatorname{Aut}\Gamma|}
 w(\Gamma;t,s).
 \end{aligned}
\end{equation}
\end{theorem}

\begin{proof}
We use the independent parameters $z_\lambda,z_B$ and set them equal to $z$ at the end. Write $\omega^{\mathrm{ref}}_{h,n}$ for the correlators of $(\lambda_{\mathrm{ref}},B_{\mathrm{ref}})$. After contour integration, they give the background or Gaussian tensors. The free energy at $(z_\lambda,z_B)=(0,0)$ is $F_g(\mathcal B_s)+F_g(G_t)$. We prove the vacuum formula together with the following identity for $n\ge1$, $2g-2+n>0$, and fixed ordinary-leaf indices $I_1,\ldots,I_n$:
\begin{equation}\label{sew:integrated-graph-sum}
 \left(\prod_{j=1}^n\mathcal C_{I_j}\right)
 \omega_{g,n}(z_\lambda,z_B)
 =\sum_{\substack{[\Gamma]\in\mathfrak G'_{g,n}\\
                         I_{l_j}=I_j\ (1\le j\le n)}}
 \frac{z_\lambda^{|S(\Gamma)|}z_B^{|E(\Gamma)|}}
      {|\operatorname{Aut}\Gamma|}w(\Gamma;t,s).
\end{equation}

First keep $B_{\mathrm{ref}}$ fixed. The source part of~\eqref{sew:marked-variation} gives, for stable $(h,r)$ with $r\ge1$,
\[
 \omega_{h,r}(\lambda_{z_\lambda},B_{\mathrm{ref}})
 =\sum_{\ell\ge0}\frac{(-z_\lambda)^\ell}{\ell!}
       \mathsf A^\ell\omega^{\mathrm{ref}}_{h,r+\ell}.
\]
Since $\mathsf A=\sum_I d_I\mathcal C_I$, every additional argument becomes a source leaf with weight $-z_\lambda d_I$. The factor $1/\ell!$ removes the order in which the source leaves were added. For the vacuum,~\eqref{sew:closed-variation} gives
\[
 F_g(\lambda_{z_\lambda},B_{\mathrm{ref}})
       -F_g(\lambda_{\mathrm{ref}},B_{\mathrm{ref}})
 =\sum_{\ell\ge1}\frac{(-z_\lambda)^\ell}{\ell!}
       \mathsf A^\ell\omega^{\mathrm{ref}}_{g,\ell}.
\]
These are the connected graphs with one vertex, at least one source leaf, and no edges.

Next keep $\lambda_{z_\lambda}$ fixed and vary $B_{z_B}$. Apply the contour operators to all remaining variables in~\eqref{sew:marked-variation}. Every $\mathsf L_i$ term vanishes by $\mathcal C_Ie_J=0$. Lemma~\ref{sew:cycles} also gives $\mathcal C_I B_{z_B}=e_I$ and hence $\mathcal C_J\mathcal C_I B_{z_B}=0$, so a bidifferential integrated in both variables contributes zero. The two terms acted on by $\mathsf Q/2$ join two ordinary leaves into an edge of weight $z_B k_{IJ}$. If the leaves lie on one connected graph, the new edge raises its genus by one. If they lie on two connected graphs, the new edge joins them and adds their genera. Cutting the new edge reverses these operations and gives exactly the connected and product terms in~\eqref{sew:marked-variation}; the case with no ordinary leaves is~\eqref{sew:closed-variation}.

This edge construction determines the Taylor coefficients simultaneously for all stable $(g,n)$. The coefficient of $z_B^{j+1}$ is $1/(j+1)$ times the coefficient of $z_B^j$ on the right-hand side of the variation equation. That coefficient uses only terms already determined at lower powers of $z_B$, starting from the values at $z_B=0$. Both sides of~\eqref{sew:integrated-graph-sum}, and the corresponding vacuum formula, therefore have the same Taylor expansion. The factorial denominators remove the order of edge insertions, and the factor $1/2$ removes the order of the two ends of each edge. Together with the symmetry of the vertex tensors and the source factorials, these give $1/|\operatorname{Aut}\Gamma|$.

Expand each vertex in this edge construction by the one-form Taylor series above. A vertex of genus $h$ then has $r$ ordinary leaves or edge ends and $\ell$ source leaves, hence total valence $r+\ell$. The edge construction starts with stable $(h,r)$. Lemma~\ref{sew:source-cancellation} allows us to include the additional genus-zero vertices with $r\le2$: their sums over source indices vanish. A genus-one vertex with $r=0$ is isolated and has graph genus one. The single-vertex graphs of genus at least two with at least one source leaf were included in the expansion of $F_g(\lambda_{z_\lambda},B_{\mathrm{ref}})$. Thus~\eqref{sew:integrated-graph-sum} sums all of $\mathfrak G'_{g,n}$, and the vacuum formula sums all of $\mathfrak G'_g$, with background and Gaussian tensors at every vertex. Setting $z_\lambda=z_B=z$ proves~\eqref{sew:graph-sum}.
\end{proof}

\begin{corollary}[The full free-energy graph sum]\label{sew:full-graph}
For $g\ge2$, as a formal series in $z$,
\begin{equation}\label{sew:full-graph-sum}
 F_g(\lambda_z,B_z)
 =\sum_{[\Gamma]\in\mathfrak G_g}
 \frac{z^{|S(\Gamma)|+|E(\Gamma)|}}{|\operatorname{Aut}\Gamma|}
 w(\Gamma;t,s).
\end{equation}
\end{corollary}

\begin{proof}
A vertex of valence zero in a connected graph is its only vertex. Thus $\mathfrak G_g\setminus\mathfrak G'_g$ consists of two graphs: a single genus-$g$ vertex of type $A$ or of type $B$. Their weights are $F_g(G_t)=C_gt^{2-2g}$ and $F_g(\mathcal B_s)$, respectively. Both have trivial automorphism group and no sources or edges, so their contributions are independent of $z$. Adding them to~\eqref{sew:graph-sum} proves the formula.
\end{proof}

Grouping connected local graphs into vertex correlators is the operation used in Eynard's intersection formula \cite[Theorem~4.1 and Appendix~A.3]{EynardIntersection}. Here the two vertex theories are the background and Gaussian neck, and the source leaves record the one-form variation.

The coefficient of $z$ comes from three kinds of graphs: a single vertex with one source, a single vertex with one loop, and two vertices joined by one edge. Their contributions are the source term, the connected two-point term, and the product of one-point terms in~\eqref{sew:closed-variation} at $z=0$.

Define the interaction by
\[
 \mathcal I(z;t,s;\hbar)
 =\sum_{g\ge2}\hbar^{2g-2}
   \bigl(F_g(\lambda_z,B_z)-F_g(\mathcal B_s)-F_g(G_t)\bigr).
\]
Set
\[
 Z_{\mathrm{neck}}(z;t,s;\hbar)
 =Z_G(t;\hbar)\exp\mathcal I(z;t,s;\hbar).
\]
Exponentiating the definition of $\mathcal I$ gives the factorization
\begin{equation}\label{sew:partition-factorization}
 Z(\lambda_z,B_z;\hbar)
 =Z_{\mathrm{back}}(s;\hbar)Z_{\mathrm{neck}}(z;t,s;\hbar),
\end{equation}
as a formal identity in $z$ and $\hbar$, where $Z_{\mathrm{back}}(s;\hbar)=Z(\mathcal B_s;\hbar)$. The Gaussian factor $Z_G$ is universal; the neck factor also includes the interaction with the background.

The interaction sums the graphs in $\mathfrak G'_g$. To write its differential-operator expression, let $J_I$ be formal variables for the half-edge indices. Write $J_B=\sum_{I\in\mathfrak I_B}J_I\mathbf v_I$ and $J_A=\sum_{I\in\mathfrak I_A}J_I\mathbf v_I$, put $J=(J_B,J_A)$, and set
\begin{equation}\label{sew:vertex-potential}
 \mathcal V(J)=
 \sum_{\substack{h\ge0,\ n\ge1\\2h-2+n>0}}
 \frac{\hbar^{2h-2+n}}{n!}
 \left(
   \widehat{\mathcal B}_{h,n}(J_B^{\otimes n};s)
   +\mathcal A_{h,n}(J_A^{\otimes n};t)
 \right).
\end{equation}
For $n\ge1$ and $2h-2+n>0$, let $\mathcal V_{h,n}(J)$ be the corresponding summand in~\eqref{sew:vertex-potential}, and let $\omega^{\mathrm{ref}}_{h,n}$ be the correlator of $(\lambda_{\mathrm{ref}},B_{\mathrm{ref}})$. Symmetry of the tensors cancels the factor $1/n!$ under differentiation, so their contour-integral definitions give
\[
 \left.
 \partial_{J_{I_1}}\cdots\partial_{J_{I_n}}
 \mathcal V_{h,n}(J)\right|_{J=0}
 =\hbar^{2h-2+n}
   \left(\prod_{j=1}^n\mathcal C_{I_j}\right)
   \omega^{\mathrm{ref}}_{h,n}.
\]
Both sides vanish when neck and background indices are mixed, since the reference data separate these two parts. Each $\mathcal C_{I_j}$ integrates the $j$th correlator variable, whereas $\partial_{J_{I_j}}$ selects the corresponding tensor argument in the generating function. The coefficients $d_I(t,s)$ and $k_{IJ}(t,s)$ are independent of $J$.

Let $[\cdot]_{\ge2}$ retain the powers $\hbar^{2g-2}$ with $g\ge2$. With the preceding correspondence, the graph sum interaction $\mathcal I$ becomes
\begin{equation}\label{sew:graph-operator}
 \mathcal I(z;t,s;\hbar)=
 \left[
 \log\left(
  \left.
  \exp\left(
       -z\hbar^{-1}\sum_I d_I\partial_{J_I}
       +\frac z2\sum_{I,J}k_{IJ}\partial_{J_I}\partial_{J_J}
      \right)e^{\mathcal V(J)}
  \right|_{J=0}
 \right)
 \right]_{\ge2}.
\end{equation}
The quadratic contraction operator also appears in Givental's quantization~\cite[Proposition~7.3]{Givental}; its relation to topological-recursion graph sums is developed by Dunin-Barkowski--Orantin--Shadrin--Spitz~\cite[Theorem~3.7]{DOSS}. In~\eqref{sew:graph-operator}, expanding $e^{\mathcal V}$ chooses a collection of background and Gaussian vertices. The identities
\[
 \begin{aligned}
 e^{-\mathcal V}\partial_{J_I}e^{\mathcal V}
 &=\partial_{J_I}\mathcal V,\\
 e^{-\mathcal V}\partial_{J_I}\partial_{J_J}e^{\mathcal V}
 &=\partial_{J_I}\partial_{J_J}\mathcal V
   +(\partial_{J_I}\mathcal V)(\partial_{J_J}\mathcal V)
 \end{aligned}
\]
show how sources and edges act. A source selects one tensor argument and multiplies by $-z\hbar^{-1}d_I$. The two edge derivatives select arguments on one vertex or on two different vertices, giving a loop or an edge between them, with coefficient $zk_{IJ}/2$ in the summed operator. In each case the selected arguments remain as indices of the vertex tensors in $w(\Gamma;t,s)$.

Setting $J=0$ requires every half-edge to be used; taking the logarithm retains connected graphs. The factorial denominators in the exponential expansions remove the temporary orderings of vertices, sources, and edges, while the factor $1/n!$ in~\eqref{sew:vertex-potential} removes the ordering of half-edges at each vertex. The factor $1/2$ for each edge removes the order of its ends. For each decorated graph the resulting coefficient is $1/|\operatorname{Aut}\Gamma|$. A graph with $S$ sources and $E$ edges has $\sum_v n_v=2E+S$, so its power of $\hbar$ is
\[
 \sum_v(2h_v-2+n_v)-S
 =2\Bigl(\sum_vh_v+E-|V(\Gamma)|+1\Bigr)-2
 =2g(\Gamma)-2.
\]
All identities in this subsection are formal in $z$.

The $t$-order of a series is its order of vanishing at $t=0$. We next prove positive $t$-order of every interaction graph and finiteness of the graph contributions at each fixed genus and power of $t$. These properties permit coefficientwise specialization at $z=1$. A fixed-genus analytic argument then identifies the resulting series with the original spectral free energies.

\Needspace{6\baselineskip}
\subsection{Positive order and convergence}
\label{sec:completion}

\subsubsection{Positive order of the Gaussian residue moments}
\label{subsubsec:positive-gaussian-order}

\begin{proposition}\label{comp:positivity}
Every stable Gaussian residue moment is divisible by $t$:
\begin{equation}\label{comp:neck-positive}
 \mathcal A_{h,n}(m_1,\ldots,m_n;t)\in t\mathbb Q[t],
 \qquad n>0,\qquad 2h-2+n>0.
\end{equation}
More precisely, a nonzero residue moment is a rational multiple of $t^{D_0-\chi_0}$, where
\[
 D_0=\frac12\sum_{i=1}^n m_i,
 \qquad \chi_0=2h-2+n,
 \qquad D_0-\chi_0\ge1.
\]
\end{proposition}

\begin{proof}
In~\eqref{gau:finite-moments} every nonzero summand has exponent $D_0-\chi_0=f+\sum_i j_i\ge1$, as proved in \eqref{gau:positive-order}.
\end{proof}

\subsubsection{Finiteness at a fixed genus and power of the period}
\label{subsubsec:graph-finiteness}

For a graph $\Gamma\in\mathfrak G'_g$, put $\chi_v=2h_v-2+n_v$ at each vertex. Let $V_A,V_B$ be the sets of neck and background vertices, and let $S_A,S_B$ denote the numbers of sources on them. Let $E_{AA},E_{AB},E_{BB}$ count edges with two neck ends, one end of each type, and two background ends, respectively. Put
\[
 \chi_A=\sum_{v\in V_A}\chi_v,\qquad
 \chi_B=\sum_{v\in V_B}\chi_v,\qquad
 D=\frac12\sum_{v\in V_A}\sum_{f\in H(v)}m_f.
\]
If a term in the product of source and edge coefficients has total $t$-order $R\ge0$ and contributes to $t^k$ in genus $g$, then
\begin{equation}\label{comp:euler-orders}
 k=D-\chi_A+R,
 \qquad \chi_A+\chi_B=2g-2+S_A+S_B.
\end{equation}

\begin{proposition}\label{comp:finiteness}
For fixed $g\ge2$ and $k\ge0$, the coefficient $[\hbar^{2g-2}t^k]\mathcal I(z;t,s;\hbar)$ is a polynomial in $z$, obtained from finitely many decorated graph contributions in~\eqref{sew:graph-sum}. It depends on finitely many background tensor components and finitely many Taylor coefficients in $t$ of the source and edge coefficients, and is holomorphic in $s$.
\end{proposition}

\begin{proof}
By~\eqref{sew:low-modes}, each neck source contributes at least $3/2$ to $D+R$. A neck--neck edge contributes at least $1$, a mixed edge at least $1/2$, and each background source or background--background edge at least $1$. Therefore
\begin{equation}\label{comp:source-bound}
 \begin{aligned}
 D+R&\ge\tfrac32S_A+E_{AA}+\tfrac12E_{AB}+S_B+E_{BB},\\
 k&\ge\tfrac12S_A-(2g-2)+\chi_B+E_{AA}+\tfrac12E_{AB}+E_{BB}.
 \end{aligned}
\end{equation}
Thus $S_A\le2k+4g-4$. Proposition~\ref{comp:positivity} gives $R\le k$, hence $S_B,E_{BB}\le k$ and $S_A+S_B\le3k+4g-4$. These bounds also hold for graphs consisting entirely of background vertices. Stability gives $\chi_v\ge1$, so \eqref{comp:euler-orders} bounds the number of vertices. The genus identity bounds the vertex genera and the number of edges. There are therefore only finitely many underlying graphs. Moreover,
\begin{equation}\label{comp:mode-bound}
 D\le k+\chi_A\le k+2g-2+S_A+S_B.
\end{equation}
This inequality bounds every positive integer $m_f$ at a neck flag, and $R\le k$ bounds the Taylor orders of the source and edge coefficients. Each background correlator of fixed genus and valence has finite pole order in each variable at finitely many ramification points. Its contour integral against $\xi_\alpha^m/m$ therefore vanishes for all sufficiently large $m$. Thus only finitely many background indices and tensor components contribute. Each graph carries the factor $z^{|S(\Gamma)|+|E(\Gamma)|}$, so the resulting finite sum is polynomial in $z$, with coefficients holomorphic in $s$.
\end{proof}

\subsubsection{The finite part and the neck factor}
\label{subsubsec:finite-part}

Define
\[
 \mathcal I_g(z;t,s)=[\hbar^{2g-2}]\mathcal I(z;t,s;\hbar).
\]
Every graph in $\mathfrak G'_g$ with a neck vertex contributes $O(t)$ by Proposition~\ref{comp:positivity}. A graph in this set with only background vertices contains a source leaf or an edge, whose coefficient is $O(t)$ by \eqref{sew:kernel-orders} and~\eqref{sew:couplings}. All remaining factors have nonnegative order. Proposition~\ref{comp:finiteness} therefore gives
\begin{equation}\label{comp:interaction-positive}
 \mathcal I_g(z;t,s)\in t\mathcal O_s[z][[t]].
\end{equation}
Each coefficient in $t$ is therefore a polynomial in $z$ and can be evaluated at $z=1$. We may now define
\[
 \begin{aligned}
 \mathcal I(t,s;\hbar)&=\mathcal I(1;t,s;\hbar),
 &\mathcal I_g(t,s)&=\mathcal I_g(1;t,s),\\
 I_{g,k}(s)&=[t^k]\mathcal I_g(t,s),
 &Z_{\mathrm{neck}}(t,s;\hbar)&=Z_G(t;\hbar)\exp\mathcal I(t,s;\hbar).
 \end{aligned}
\]
Evaluating~\eqref{comp:interaction-positive} at $z=1$ gives $\mathcal I_g(t,s)\in t\mathcal O_s[[t]]$. We next identify this series with the analytic free energy after subtracting its Gaussian and background terms.

\subsubsection{Holomorphy at the node}
\label{subsubsec:analytic-completion}

\begin{proposition}\label{comp:analytic}
Let $\mathcal S_{t,s}$ be a family satisfying Definition~\ref{def:regular-node}, and let $g\ge2$. After shrinking the parameter domain, the remainder
\[
 F_g(\mathcal S_{t,s})-
 C_gt^{2-2g}-F_g(\mathcal B_s)
\]
extends holomorphically across $t=0$, jointly in $s$. Its Taylor series is $\sum_{k\ge1}I_{g,k}(s)t^k$.
\end{proposition}

\begin{proof}
On the cover $t=\epsilon^2$, put $U=\epsilon\zeta$ and $V=\epsilon/\zeta$, and use fixed discs near $\zeta=\pm1$. Let $\lambda_{z,A}$ denote the restriction of $\lambda_z$ to the neck. Its scaled Gaussian part is
\[
 \Lambda_G=\epsilon^{-2}\lambda_G
 =\frac12(\zeta-\zeta^{-1})^2\frac{d\zeta}{\zeta},
\]
which is independent of $\epsilon$. The bracket in~\eqref{sew:numerator} is holomorphic in $(U,V,\tau,s)$ and vanishes at $U=V=\tau=0$, because the period normalization gives $c(0,s)=M_0(0,s)/2$. Since $\tau=O(\epsilon^2)$, substituting $U=\epsilon\zeta$, $V=\epsilon/\zeta$ makes this bracket divisible by $\epsilon$. Consequently
\begin{equation}\label{comp:scaled-interpolation}
 \epsilon^{-2}\lambda_{z,A}
 =\Lambda_G+z\epsilon R_A(\zeta;\epsilon,s),
\end{equation}
where $R_A$ is a holomorphic one-form retaining the double zeros at $\zeta=\pm1$. On each background disc, $\lambda_z=\lambda_{\mathrm{ref}}+z\epsilon^2 R_B$ with $R_B$ holomorphic. The normally convergent sewing formulas give bidifferential coefficients holomorphic in $(\epsilon,s)$: the mixed correction is $zO(\epsilon)$ and both diagonal corrections are $zO(\epsilon^2)$.

At a neck ramification point use a fixed Morse coordinate $\xi$ for $\zeta+\zeta^{-1}$, and at a background ramification point use its chosen Morse coordinate. Write the scaled neck one-form, or the unscaled background one-form, as
\[
 \bigl(b_a(z;\epsilon,s)+\mu_a(z;\epsilon,s)\xi
       +O(\xi^2)\bigr)\xi\,d\xi.
\]
The coefficient $\mu_a$ is the leading odd coefficient entering the recursion denominator. Regularity of the Gaussian and background data, together with~\eqref{comp:scaled-interpolation}, gives
\begin{equation}\label{comp:leading-odd-coefficient}
 \mu_a(z;\epsilon,s)
 =\mu_a^{(0)}(s)+z\epsilon^{r_a}\nu_a(\epsilon,s),
 \qquad \mu_a^{(0)}(s)\ne0,
 \qquad
 r_a=\begin{cases}1,&a\text{ in the neck},\\2,&a\text{ in the background}.
 \end{cases}
\end{equation}
All the functions here are holomorphic. After shrinking the $\epsilon$-disc and spectator polydisc, the coefficients $\mu_a$ stay bounded away from zero for complex $|z|<2$.

Let $F_g(z;\epsilon,s)=F_g(\lambda_z,B_z)$ after pullback to this cover. At fixed $(g,n)$, residue recursion uses finitely many Taylor coefficients at the ramification points and inverse powers of $\mu_a$. This follows by induction: lower correlators have finite pole orders, so taking the next residue uses a finite Taylor polynomial. Each of the $2g-2+n$ recursion kernels contributes at most $\epsilon^{-2}$, while the bidifferential coefficients are holomorphic in $\epsilon$. In~\eqref{eq:closed-convention}, choose the local primitive $\Phi_a$ of $\lambda_z$ with $\Phi_a(a)=0$. This leaves $F_g$ unchanged because $\omega_{g,1}$ has zero residue at $a$, and $\Phi_a$ is holomorphic in the parameters. Hence $\epsilon^{4g-2}F_g(z;\epsilon,s)$ is jointly holomorphic for $|z|<2$.

Write its convergent Laurent expansion as
\[
 F_g(z;\epsilon,s)=\sum_{j\ge2-4g}A_j(z;s)\epsilon^j.
\]
Each $A_j$ is polynomial in $z$. Indeed, the input Taylor coefficients are affine in $z$, and for each positive integer $m$ the inverse powers in the recursion expand as
\[
 \mu_a^{-m}
 =(\mu_a^{(0)})^{-m}
 \sum_{q\ge0}\binom{-m}{q}
 \left(\frac{z\epsilon^{r_a}\nu_a}{\mu_a^{(0)}}\right)^q.
\]
Only finitely many $q$ can contribute to a fixed power of $\epsilon$, since $r_a\ge1$. The finite residue calculation therefore gives the asserted polynomial dependence.

We identify these polynomials using the formal graph identity. Fix a sufficiently small circle $|\epsilon|=\delta>0$. Laurent-coefficient extraction on this circle commutes with Taylor-coefficient extraction in $z$ near zero. At each fixed power of $z$, the graph identity has finitely many underlying graphs, and their index sums converge normally on the contour products, locally uniformly along this circle, by~\eqref{sew:coupling-primitives}. We may therefore extract their Laurent coefficients term by term. Proposition~\ref{comp:finiteness} says that at each fixed power of $t=\epsilon^2$ the result is the finite polynomial supplied by the graph sum. Thus every Taylor coefficient at $z=0$ of $A_j(z;s)$ agrees with that polynomial. Since both are polynomials, they agree for all $z$, in particular at $z=1$.

The free energy at $z=0$ is $F_g(\mathcal B_s)+C_g\epsilon^{4-4g}$. By~\eqref{comp:interaction-positive}, the remaining graph coefficients vanish at all nonpositive powers of $\epsilon$, and at every odd power. At $z=1$ the remainder is therefore jointly holomorphic in $(\epsilon,s)$, even in $\epsilon$, and zero at $\epsilon=0$. It is consequently jointly holomorphic in $(t,s)$, with Taylor series $\sum_{k\ge1}I_{g,k}(s)t^k$.
\end{proof}

\subsection{The factorization theorem}
\label{sec:factorization}

Write $\mathbb C\{t,s-s_0\}$ for the ring of power series converging jointly in a neighborhood of $(0,s_0)$.

\begin{theorem}[Normalization and neck factorization]\label{thm:factorization}
For a family satisfying Definition~\ref{def:regular-node},
\begin{equation}\label{eq:main-factorization}
 Z(\mathcal S_{t,s};\hbar)
 =Z(\mathcal B_s;\hbar)Z_G(t;\hbar)
       \exp\mathcal I(t,s;\hbar).
\end{equation}
The interaction $\mathcal I$ is the connected stable graph sum of Theorem~\ref{sew:graph}, evaluated coefficientwise at $z=1$, with positive valence at every vertex. Its vertex weights are the background tensors~\eqref{sew:background-tensor} and the finite Gaussian tensors of Proposition~\ref{prop:gaussian-moments}; its source and edge weights are $-d_I$ and $k_{IJ}$ from~\eqref{sew:couplings}. Each coefficient $[\hbar^{2g-2}t^k]\mathcal I$ is a finite sum over graphs and half-edge index assignments. For every $g\geq2$,
\begin{equation}\label{eq:main-free-energy}
 F_g(\mathcal S_{t,s})
 =C_gt^{2-2g}+F_g(\mathcal B_s)
     +\sum_{k\geq1}I_{g,k}(s)t^k,\qquad
 \sum_{k\geq1}I_{g,k}(s)t^k\in t\mathbb C\{t,s-s_0\}.
\end{equation}
The Taylor series is jointly convergent at fixed genus.
\end{theorem}

\begin{proof}
Proposition~\ref{comp:analytic} gives~\eqref{eq:main-free-energy}, with the finite graph coefficients of Proposition~\ref{comp:finiteness}. Multiplying by $\hbar^{2g-2}$, summing over $g\ge2$, and exponentiating gives~\eqref{eq:main-factorization} as a formal identity in $\hbar$.
\end{proof}

The universal factor $Z_G$ is the isolated Gaussian vacuum. The neck factor $Z_{\mathrm{neck}}=Z_G\exp\mathcal I$ includes the coupling to the background and depends on the germ of the spectral family. The partition identity is formal in $\hbar$; the interaction series converges at each fixed genus $g\ge2$.

\section{Gromov--Witten applications at the conifold}
\label{sec:gw}

We illustrate Theorem~\ref{thm:factorization} with the mirror curves of $K_{\mathbb P^2}$, $K_{\mathbb F_0}$, and $K_{\mathbb F_1}$. Here $K_S$ denotes the total space of the canonical bundle of $S$. These examples give explicit background contributions coupled to the universal relative-$\mathbb P^1$ neck. We specify the families in the toric and framing conventions of the remodeling theorem of Fang--Liu--Zong \cite{FLZ}, apply the spectral factorization, and identify its background and neck factors. Remodeling gives Gromov--Witten factorizations, using the comparison of Guo--Xu--Zhang \cite{GXZ} for the $X_2$ background. We also identify the background contour insertions with closed GW ancestors.

\subsection{The mirror families and spectral factorization}
\label{subsec:gw-curves}

For a smooth preferred outer flag, the mirror equation begins $1+U+V$, and framing $f$ gives spectral coordinates $-\log U-f\log V,-\log V$ \cite[Sections~4.1--4.2 and~4.6]{FLZ}. We use the following three equations:
\begin{align}
 1+U+V+a\frac{U^3}{V}&=0,
 &&K_{\mathbb P^2},\label{eq:gw-p2-flz}\\
 1+U+V+\frac1{c_0^2}\frac{U^2}{V}
       +\frac{r^2}{c_0^2}U^2&=0,
 &&K_{\mathbb F_0},\label{eq:gw-f0-flz}\\
 1+U+V-\frac{q(1-q)}{c_1^2}\frac{U^2}{V}
       -\frac q{c_1^3}\frac{U^3}{V}&=0,
 &&K_{\mathbb F_1}.\label{eq:gw-f1-flz}
\end{align}
Their Newton polygons and the preferred boundary edge between the monomials $1,V$ give the indicated targets and outer flags. To describe a transverse conifold chart, put
\begin{equation}\label{eq:gw-family-parameters}
 a=-\tfrac1{27}+\delta,\qquad
 c_0=-2(1+r)-\delta,\qquad c_1=1+2q+\delta.
\end{equation}
The spectator parameter is $r$ or $q$. The letter $\delta$ denotes the smoothing parameter in all three families.

For local $\mathbb P^2$, set $U=1/Y$, $V=X/Y$ and use framing $f=-2$. For local $\mathbb F_0$, set $U=c_0/(rY)$, $V=X/(rY)$ and again use $f=-2$. The respective equations become
\begin{align}
 XY(1+X+Y)-\tfrac1{27}+\delta&=0,
 \label{eq:gw-p2-family}\\
 X+X^{-1}+r(Y+Y^{-1})-2(1+r)-\delta&=0.
 \label{eq:gw-f0-family}
\end{align}
We take the spectral coordinates
\begin{equation}\label{eq:gw-p2-f0-spectral-coordinates}
 x=2\log X-\log Y,\qquad y=\log Y-\log X.
\end{equation}
They agree with the framed spectral coordinates specified above up to parameter-dependent additive constants. These constants preserve the recursion kernel and the vanishing period. In particular, the spectral $x$-maps are fixed before the degeneration is taken.

For local $\mathbb F_1$, set $U=-c_1XY$, $V=-X^2Y$ and use framing $f=1$. We obtain
\begin{equation}\label{eq:gw-f1-family}
 Y^{-1}-X^2-(1+2q+\delta)X
          +q(1-q)-qXY=0.
\end{equation}
The spectral coordinates are, up to additive constants,
\begin{equation}\label{eq:gw-f1-spectral-coordinates}
 x=-3\log X-2\log Y+\log(1-q),\qquad
 y=-2\log X-\log Y.
\end{equation}
Thus each spectral $x$-map is fixed by an explicit outer flag and framing before the degeneration is taken.

All three framings have nonzero tangent weights at every torus fixed point. In cyclic order around the Newton polygon, the triples are
\begin{equation}\label{eq:gw-original-tangent-weights}
\begin{aligned}
 K_{\mathbb P^2}:&\quad (1,-2,1),\ (-5,3,2),\ (4,-1,-3),\\
 K_{\mathbb F_0}:&\quad (1,-2,1),\ (-3,1,2),\ (-1,2,-1),\ (3,-1,-2),\\
 K_{\mathbb F_1}:&\quad (1,1,-2),\ (3,-2,-1),\ (1,-3,2),\ (-5,2,3).
\end{aligned}
\end{equation}
For each triangle, the weights are ordered by its interior vertex and its two boundary vertices in counterclockwise order. Thus these choices lie in the generic equivariant setting of the remodeling comparison.

The central nodes are $(-1/3,-1/3)$, $(1,1)$, and $(-q,-1/q)$, respectively. We work on spectator neighborhoods satisfying
\begin{equation}\label{eq:gw-regular-spectators}
 r\ne0,-1,-\tfrac14,\qquad
 q(1-q)(1-4q)(3+4q)(1+2q)(1-q+4q^2)\ne0,
\end{equation}
with finite $r,q$. For the GW comparisons, choose small simply connected neighborhoods $D_r\ni r_0$ and $D_q\ni q_0$ satisfying \eqref{eq:gw-regular-spectators}, with $0<|r_0^{-1}|$ and $0<|q_0|$ sufficiently small. We shrink these neighborhoods within the large-radius domains of the background GW germs specified below. Thus the GW statements use
\begin{equation}\label{eq:gw-comparison-domains}
 r\in D_r,\qquad q\in D_q,
\end{equation}
near $r=\infty$ and $q=0$, respectively, where the background K\"ahler parameters are $-1/r$ and $-q(1-q)^3$.

Choose local logarithm branches consistently at the two preimages of each node. On each smooth compactified mirror curve choose the fundamental bidifferential $B^c$ with zero period along the vanishing cycle $A_c$. Fix the orientation of $A_c$ and set
\begin{equation}\label{eq:gw-spectral-period}
 t=\frac1{2\pi i}\int_{A_c}y\,dx.
\end{equation}

\begin{lemma}\label{lem:gw-regular-families}
With this polarization, the three mirror families are regular spectral-node families in the sense of Definition~\ref{def:regular-node}.
\end{lemma}
\begin{proof}
For the first two curves, the tangent slopes $\eta=dY/dX$ at the node are $(-1\pm i\sqrt3)/2$ and $\pm i/\sqrt r$. In the coordinates \eqref{eq:gw-p2-f0-spectral-coordinates}, the spectral slope is $(\eta-1)/(2-\eta)$. The condition $\eta\ne2$ follows from \eqref{eq:gw-regular-spectators}, so both branches have nonzero $dx$ and distinct $dy/dx$.

We use $w$ as a rational coordinate on each central normalization and write $w_\pm$ for the two node preimages. The first normalization and its spectral differentials are
\begin{equation}\label{eq:gw-p2-normalization}
 X=-\frac1{3w(1-w)},\qquad Y=\frac{w^2}{3(1-w)},\qquad
 dx=\frac{5w-4}{w(1-w)}\,dw,\qquad dy=\frac3w\,dw.
\end{equation}
The node preimages are $w_\pm=(1\pm i\sqrt3)/2$. The remaining ramification point $w=4/5$ is simple, $dy$ is nonzero there, and the punctures $0,1,\infty$ are distinct.

The second normalization is
\begin{equation}\label{eq:gw-f0-normalization}
 X=\frac{r+w}{w(1-w)},\qquad
 Y=-\frac{w(r+w)}{r(1-w)},\qquad
 w_\pm=\pm i\sqrt r.
\end{equation}
Direct differentiation gives
\begin{equation}\label{eq:gw-f0-differentials}
 dx=\frac{3w^2+(4r-2)w-3r}{w(1-w)(r+w)}\,dw,
 \qquad dy=\frac2w\,dw.
\end{equation}
The numerator of $dx$ has discriminant $4(4r+1)(r+1)$. Its roots are distinct and avoid the node preimages and the punctures $0,1,-r,\infty$ under \eqref{eq:gw-regular-spectators}. The differential $dy$ is nonzero at both roots.

For the third curve, the central normalization is
\begin{equation}\label{eq:gw-f1-normalization}
 X=\frac{(1-w)(1-qw)}w,\qquad
 Y=\frac{w^2}{1-w},\qquad
 w_\pm=\frac{1\pm\sqrt{1-4q}}{2q}.
\end{equation}
The two points $w_\pm$ map to $(-q,-1/q)$. The Hessian determinant of the defining Laurent polynomial at this point is $q^2(4q-1)$, so the singularity is an ordinary node. The pullbacks of the spectral differentials are
\begin{equation}\label{eq:gw-f1-differentials}
 dx=-\frac{1-(2+4q)w+5qw^2}{w(1-w)(1-qw)}\,dw,
 \qquad
 dy=\frac{1+2q-3qw}{(1-w)(1-qw)}\,dw.
\end{equation}
The numerator of $dx$ has discriminant $4(1-q+4q^2)$. Its roots avoid $w_\pm$ when $(1-4q)(3+4q)\ne0$, and $dy$ does not vanish at them under \eqref{eq:gw-regular-spectators}. The monomial coordinate transformation in \eqref{eq:gw-f1-spectral-coordinates} is invertible, so the node branches have distinct slopes. The punctures $0,1,1/q,\infty$ are distinct. The normalizations are rational, and these are the only singularities of the central compactified curves.

In the three defining equations the derivatives with respect to $\delta$ at the nodes are $1,-1,q$, respectively. Thus the smoothings are transverse. In the local form $v^2=u^2-\tau$, transversality gives $\tau=\delta$ times a holomorphic function nonzero at the origin. The residue calculation leading to \eqref{sew:period-unit} gives $t=M(0,0,s)\tau/2+O(\tau^2)$ with $M(0,0,s)\ne0$. Hence $t=\delta c(\delta,s)$ with $c(0,s)\ne0$, where $s$ denotes the spectator parameter when present. The remaining ramification points and punctures persist regularly. Finally, $B^c$ has the required zero neck period by construction.
\end{proof}

Let $\mathcal S^Y_{t,s}$ denote one of these spectral families and $\mathcal B^Y_s$ its normalized central curve, where $Y=K_{\mathbb P^2},K_{\mathbb F_0},K_{\mathbb F_1}$. Theorem~\ref{thm:factorization} now gives
\begin{equation}\label{eq:gw-spectral-factorization}
 Z(\mathcal S^Y_{t,s};\hbar)
 =Z(\mathcal B^Y_s;\hbar)Z_G(t;\hbar)
       \exp\mathcal I_Y(t,s;\hbar).
\end{equation}
For each $g\ge2$, the coefficient $\mathcal I_{Y,g}$ is a convergent series in $t$ with zero constant term; each Taylor coefficient is the finite graph sum already specified in the theorem. It remains to identify the background and Gaussian theories by mirror symmetry and Hurwitz theory.

\subsection{The normalized backgrounds}
\label{subsec:gw-backgrounds}

For a toric target $M$ with a framed outer brane, write $F_{h,n}^{M,A}$ for its connected genus-$h$ open GW potential with its $n$ boundary components ordered, as defined in \cite[Section~3.13]{FLZ}. Write $F_g^{M,A}$ for its closed potential, including degree zero, and set $Z_M^A=\exp\sum_{g\ge2}g_s^{2g-2}F_g^{M,A}$. The comparisons use their analytic germs near large radius. At fixed closed parameters, continue the open variables from the chosen brane chart to the coordinate discs on which the tensors are used. The spectral notation $\omega_{h,n},F_g,Z$ remains that of Section~\ref{sec:setup}.

The B-model graph sum in \cite[Theorem~6.8]{FLZ} uses the recursion of \cite[equation~(3-3) and Theorem~3.7]{DOSS}, whose kernel has the opposite sign to \eqref{eq:recursion-kernel}. Its stable $n$-point differentials therefore equal $(-1)^n\omega_{h,n}$. The closing denominator $2-2g$ in \cite[Definition~6.10]{FLZ} then gives the same closed free energy $F_g$. Remodeling \cite[Theorems~7.5--7.6]{FLZ} therefore gives
\begin{equation}\label{eq:gw-remodeling-signs}
 \omega_{h,n}=(-1)^{h-1}d_1\cdots d_nF_{h,n}^{M,A},
 \qquad F_g=(-1)^{g-1}F_g^{M,A}.
\end{equation}
Here $d_j$ denotes the exterior derivative in the $j$th open variable. The first identity is for $n>0$ and $2h-2+n>0$, after the open and closed mirror maps; the second is for $g\ge2$. Closed free energies use the connected graph sum with no ordinary external leaves. The disk and annulus use their separate unstable definitions.

For all three backgrounds, use the normalization coordinate $w$ from Lemma~\ref{lem:gw-regular-families} and write $U_{\mathrm b},V_{\mathrm b}$ for the background coordinates. With the framing $f$ specified in each case, set
\begin{equation}\label{eq:gw-background-coordinates}
 \chi=U_{\mathrm b}V_{\mathrm b}^{f},\qquad
 x_{\mathrm b}=-\log\chi,\qquad y_{\mathrm b}=-\log V_{\mathrm b}.
\end{equation}
The disk branch approaches $(U_{\mathrm b},V_{\mathrm b})=(0,1)$ as $\chi\to0$, and its disk form is $\log V_{\mathrm b}\,d\log\chi$. Write $\Theta_\chi=\chi\,\partial/\partial\chi$ for differentiation with respect to $\log\chi$, with the closed parameters fixed.

\subsubsection*{Local $\mathbb P^2$}
On the normalization \eqref{eq:gw-p2-normalization}, set
\begin{equation}\label{eq:gw-p2-background}
 U_{\mathrm b}=1-w,\qquad V_{\mathrm b}=w,\qquad
 U_{\mathrm b}+V_{\mathrm b}=1.
\end{equation}
This is the mirror of $\mathbb C^3$. Its preferred outer flag has coordinates $U=-U_{\mathrm b},V=-V_{\mathrm b}$, framing $f=4$, and ordered weights $(1,4,-5)$. Thus
\begin{equation}\label{eq:gw-p2-states}
 \chi=U_{\mathrm b}V_{\mathrm b}^4=w^4(1-w),\qquad
 \chi(w_\pm)=-1.
\end{equation}
The framed open coordinate is $\widehat X=UV^4=-\chi$; we make the same sign change in the A-model open parameter. The disk branch is $U_{\mathrm b}=O(\chi)$, and its disk potential has the normalization
\begin{equation}\label{eq:gw-disk-normalization}
 \chi=U_{\mathrm b}(1-U_{\mathrm b})^4,\qquad
 \Theta_\chi F_{0,1}^{\mathbb C^3,A}=\log V_{\mathrm b},\qquad
 F_{0,1}^{\mathbb C^3,A}=-\chi-\tfrac94\chi^2+O(\chi^3).
\end{equation}
The inherited spectral coordinates satisfy
\begin{equation}\label{eq:gw-p2-scale}
 \begin{gathered}
 x=x_{\mathrm b}-\log3,\qquad y=-3y_{\mathrm b}+\log(-1),\\
 \lambda=-3\log V_{\mathrm b}\,d\log\chi+\log(-1)\,dx.
 \end{gathered}
\end{equation}
Choose compatible local logarithm branches in this identity. The signed scale is $-3$. The last term in \eqref{eq:gw-p2-scale} is a constant multiple of $dx$ and does not affect the recursion kernel or the closing residues. Stable $n$-point tensors therefore carry $(-3)^{2-2h-n}$ and the closed potential carries $3^{2-2g}$.

\subsubsection*{Local $\mathbb F_0$}
On the normalization \eqref{eq:gw-f0-normalization}, set
\begin{equation}\label{eq:gw-f0-chart}
 \begin{gathered}
 U_{\mathrm b}=\frac{r(1-w)}{r+w},\qquad V_{\mathrm b}=w,\qquad Q'_R=-1/r,\\
 U_{\mathrm b}+V_{\mathrm b}-Q'_R U_{\mathrm b}V_{\mathrm b}=1.
 \end{gathered}
\end{equation}
This is the mirror of
\begin{equation}\label{eq:gw-resolved-conifold}
 R=\operatorname{Tot}_{\mathbb P^1}(\mathcal O(-1)\oplus\mathcal O(-1)).
\end{equation}
Choose the outer leg $(U_{\mathrm b},V_{\mathrm b})=(0,1)$, with coordinates $U=-U_{\mathrm b},V=-V_{\mathrm b}$, framing $f=3$, and ordered weights $(1,3,-4)$. The framed coordinate and the values at the node preimages are
\begin{equation}\label{eq:gw-f0-states}
 \begin{gathered}
 \chi=U_{\mathrm b}V_{\mathrm b}^3=\frac{rw^3(1-w)}{r+w},\\
 U_{\mathrm b}(w_\pm)=\mp i\sqrt r,\qquad
 V_{\mathrm b}(w_\pm)=\pm i\sqrt r,\qquad \chi(w_\pm)=-r^2.
 \end{gathered}
\end{equation}
Here the framed open coordinate $UV^3$ equals $\chi$. The disk branch has $U_{\mathrm b}=O(\chi)$, $V_{\mathrm b}=1+O(\chi)$ and $\Theta_\chi F_{0,1}^{R,A}=\log V_{\mathrm b}$. The inherited spectral coordinates and one-form are
\begin{equation}\label{eq:gw-f0-scale}
 \begin{gathered}
 x=x_{\mathrm b}+\log(-r^2),\qquad y=-2y_{\mathrm b}-\log(-r),\\
 \lambda=-2\log V_{\mathrm b}\,d\log\chi-\log(-r)\,dx.
 \end{gathered}
\end{equation}
The signed scale is $-2$, so stable tensors carry $(-2)^{2-2h-n}$ and the closed potential carries $2^{2-2g}$. The extra constant multiple of $dx$ has no effect on recursion or the closing residues. Remodeling applies to these backgrounds, including their closed vacua.

\subsubsection*{Local $\mathbb F_1$}
On the smooth normalization \eqref{eq:gw-f1-normalization}, the background coordinate functions are
\begin{equation}\label{eq:gw-f1-background-parametrization}
 U_{\mathrm b}=\frac{w(1-qw)}{1-q},\qquad
 V_{\mathrm b}=(1-w)(1-qw)^2,\qquad Q=-q(1-q)^3.
\end{equation}
Their image satisfies the plane equation
\begin{equation}\label{eq:gw-f1-background}
 V_{\mathrm b}^2+\bigl((1-q)(1+2q)U_{\mathrm b}-1\bigr)V_{\mathrm b}
       +Q U_{\mathrm b}^2(1-U_{\mathrm b})=0.
\end{equation}
The background is the smooth punctured curve $\mathbb P^1_w\setminus\{0,1,1/q,\infty\}$. The map $w\mapsto(U_{\mathrm b}(w),V_{\mathrm b}(w))$ identifies the two distinct points $w_\pm$, producing the node $(1/(1-q),-q)$ in its plane image \eqref{eq:gw-f1-background}. The spectral data are defined on the smooth curve with these two points separate. At $q=0$, the disk branch reduces to $U_{\mathrm b}+V_{\mathrm b}=1$. With framing $f=1$, the chosen spectral coordinates pull back to
\begin{equation}\label{eq:gw-f1-projection}
 \begin{gathered}
 \chi=U_{\mathrm b}V_{\mathrm b}=\frac{w(1-w)(1-qw)^3}{1-q},\\
 x=x_{\mathrm b},\qquad y=y_{\mathrm b},\qquad
 \chi(w_\pm)=-\frac q{1-q}.
 \end{gathered}
\end{equation}
These are the spectral data of the background
\begin{equation}\label{eq:gw-f1-target}
 X_2=\operatorname{Tot}_{\mathbb P^1}(\mathcal O(1)\oplus\mathcal O(-3)),
 \qquad Q=-q(1-q)^3.
\end{equation}
Define its closed equivariant theory by localization, including $d=0$:
\begin{equation}\label{eq:gw-f1-definition}
 F_g^{X_2,A}(Q)=\sum_{d\ge0}Q^d
 \int_{[\overline{\mathcal M}_{g,0}(\mathbb P^1,d)]^{\mathrm{vir}}}
 e_T\!\left(-R^\bullet\pi_*f^*
 [\mathcal O(1)\oplus\mathcal O(-3)]\right).
\end{equation}
The weights, ordered as base, $\mathcal O(-3)$, $\mathcal O(1)$, are $(-2,1,1)$ at one fixed point and $(2,-5,3)$ at the other. Choose the $\mathcal O(-3)$ outer leg at the first point, with ordered flag weights $(1,1,-2)$. Its open theory is defined by formal-relative GW theory \cite{LLLZ}, or equivalently by open localization with framed disk factors \cite{FangLiu}; simplify the rational equivariant expressions before restricting the weights. The disk branch is
\begin{equation}\label{eq:gw-f1-states}
 \Theta_\chi F_{0,1}^{X_2,A}=\log V_{\mathrm b},\qquad
 U_{\mathrm b}=O(\chi),\qquad w=O(\chi).
\end{equation}
The sign of the open coordinate is fixed by this formula. The inherited data in \eqref{eq:gw-f1-projection} have scale $1$.

Mari\~no \cite{MarinoLocalCurves} proposed the spectral description of $X_2$, and Eynard \cite[Sections~3.6.1 and~4.4.2]{EynardPartitions} proved the closed correspondence. Guo--Xu--Zhang \cite[Theorems~2.3 and~2.7]{GXZ} established the open and closed comparison, which takes the following form in our conventions.

\begin{proposition}[Remodeling for the $X_2$ background]\label{gw:f1}
The equivariant open and closed $X_2$ theories just specified satisfy remodeling for the spectral data \eqref{eq:gw-f1-projection}, as identities of analytic germs near large radius, with closed mirror map \eqref{eq:gw-f1-target} and disk branch \eqref{eq:gw-f1-states}. In particular, the stable open and closed identities are \eqref{eq:gw-remodeling-signs}, with the corresponding unstable disk and annulus identities.
\end{proposition}
\begin{proof}
Exchange the line-bundle summands and use $p=-2$, $f=-2$ in the formulas of \cite[Sections~2.2--2.4 and~3.3]{GXZ}. Their comparison extends coefficientwise to these parameters by rationality in $p,f$, and their nonvanishing condition has value $30$. Writing $x_*,y_*,\mathbf Q,\widehat X_*,X_*$ for their coordinates, we obtain
\begin{equation}\label{eq:gw-x2-comparison-coordinates}
 \begin{gathered}
 x_*=-x+\log(1-q),\qquad y_*=y,\qquad \mathbf Q=-Q,\\
 \widehat X_*=(1-q)\chi,\qquad X_*=-\chi.
 \end{gathered}
\end{equation}
Reversing the normal-weight order from $(1,-2,1)$ to $(1,1,-2)$ changes a winding-$m$ disk factor by $(-1)^{m-1}$; together with $X_*=-\chi$, this gives $(-1)^n$ for $n$ boundaries. The change $x\mapsto-x$ compensates for their opposite recursion-kernel and closing-residue signs, giving \eqref{eq:gw-remodeling-signs} and the disk and annulus normalizations. The spectral expressions are analytic near large radius, so these identities hold as analytic germs.
\end{proof}

For $g\ge2$, put
\begin{equation}\label{eq:gw-constants}
 A_g=\frac{(-1)^g|B_{2g}B_{2g-2}|}{2(2g)(2g-2)(2g-2)!},
 \qquad a_g=\frac{|B_{2g}|}{2g(2g-2)!}.
\end{equation}
The closed background potentials are
\begin{equation}\label{eq:gw-background-vacua}
 F_g^{\mathbb C^3,A}=A_g,\qquad
 F_g^{R,A}(Q)=2A_g+a_g\operatorname{Li}_{3-2g}(Q),\qquad g\ge2.
\end{equation}
The degree-zero terms follow from the Hodge integral in \cite[Theorem~4]{FP}, and the resolved-conifold multiple-cover formula is \cite[Theorem~3]{FP}. Since $\operatorname{Li}_{3-2g}(Q^{-1})=\operatorname{Li}_{3-2g}(Q)$, its closed background may be written at $Q=-r$; the open chart remains at $Q'_R=-1/r$.

\subsection{The universal neck as Hurwitz theory}
\label{subsec:gw-neck}

The covers of $\mathbb P^1$ associated with Gaussian topological recursion are branched only over $0,1,\infty$, with ramification index $2$ at every point over $1$ \cite[Sections~3--4, Theorem~4.3]{DMSS}. Their Gaussian matrix-model interpretation is given in \cite[Section~2.3]{dMKR}. We express this correspondence in the residue basis \eqref{gau:end-moments}.

Let $H^\circ_{h,d}(\mu,\rho,\nu)$ denote the connected three-profile Hurwitz number of a parametrized $\mathbb P^1$, with the branch profiles over $0,1,\infty$ unlabelled and each cover weighted by the inverse automorphism order. When
\begin{equation}\label{eq:gw-hurwitz-dimension}
 2h-2-d+\ell(\mu)+\ell(\rho)+\ell(\nu)=0,
\end{equation}
there are no further branch points, and this is the relative invariant
\begin{equation}\label{eq:gw-relative-invariant}
 H^\circ_{h,d}(\mu,\rho,\nu)
 =\left\langle\mu,\rho,\nu\right\rangle^{
 \mathbb P^1,\operatorname{rel}\{0,1,\infty\},\circ}_{h,d}.
\end{equation}
This is the case without descendant insertions of the relative-GW/Hurwitz correspondence of Okounkov--Pandharipande \cite[Section~1.6]{OP}.

For positive integers $m_i$, retain $c(m,j)$ from \eqref{gau:chebyshev}, and write
\begin{equation}\label{eq:gw-hurwitz-indices}
 \begin{gathered}
 k_i=m_i-2j_i>0,\qquad 2E=\sum_i k_i,\qquad J_0=\sum_i j_i,\\
 \mu=(k_1,\ldots,k_n),\qquad
 z_\mu=\prod_{a\ge1}a^{m_a(\mu)}m_a(\mu)!.
 \end{gathered}
\end{equation}
Here $\mu$ denotes the partition with these parts and $m_a(\mu)$ their multiplicities.

\begin{proposition}[Hurwitz expression for the neck]\label{prop:gw-neck-identification}
For $h\ge0$, $n>0$, and $2h-2+n>0$, the Gaussian tensor of Proposition~\ref{prop:gaussian-moments} is
\begin{equation}\label{eq:gw-neck-tensor}
 \begin{split}
 \mathcal A_{h,n}(\mathbf m;t)
 =\frac1{\prod_i m_i}
 \sum_{\substack{0\le j_i\le\lfloor m_i/2\rfloor\\k_i>0}}
 &\left(\prod_i c(m_i,j_i)\right)z_\mu\\[-2pt]
 {}\times&\sum_{\substack{\nu\vdash2E\\\ell(\nu)=2-2h-n+E}}
 t^{\ell(\nu)+J_0}
 H^\circ_{h,2E}(\mu,(2^E),\nu).
 \end{split}
\end{equation}
The expression is zero if $\sum_i m_i$ is odd. Every nonzero term has positive $t$-order. The tensor is the same for either choice of neck end at each insertion.
\end{proposition}
\begin{proof}
Fix the incoming permutation with its labelled cyclic slots, as in \eqref{gau:pairing-count}. The conjugacy class of profile $\mu$ has $(2E)!/z_\mu$ elements, whereas the weighted cover count divides the number of transitive permutation triples by $(2E)!$. Consequently
\begin{equation}\label{eq:gw-pairing-hurwitz}
 N_{h,n}(\mathbf k)=z_\mu
 \sum_{\substack{\nu\vdash2E\\\ell(\nu)=2-2h-n+E}}
 H^\circ_{h,2E}(\mu,(2^E),\nu).
\end{equation}
The profile $(2^E)$ records the pairing, and the final profile records its faces. The displayed length condition is both the Euler relation for the ribbon graph and \eqref{eq:gw-hurwitz-dimension}. Substitute \eqref{eq:gw-pairing-hurwitz} into \eqref{gau:finite-moments}. Its power of $t$ is $\ell(\nu)+J_0\ge1$. Equality at the two ends follows from the Gaussian residue identity in Subsection~\ref{subsubsec:gaussian-moments} and Lemma~\ref{sew:cycles}.
\end{proof}

Thus the same Gaussian neck, including all its stable insertion tensors, has a relative-GW interpretation for every family under consideration. Its two ends specify the two contour functionals; the three relative target points encode the incoming cycles, pairings, and faces.

To express the closed statement in A-model variables, set
\begin{equation}\label{eq:gw-phase}
 t_A=i t,\qquad \hbar=i g_s.
\end{equation}
For $g\ge2$, the isolated neck vacuum becomes
\begin{equation}\label{eq:gw-universal-vacuum}
 \begin{gathered}
 U_g^A(t_A)=C_gt_A^{2-2g}
 =\lim_{\varepsilon\to0}\varepsilon^{2g-2}
       F_g^{R,A,+}(e^{i\varepsilon t_A}),\\
 Z_{\mathrm{univ}}^A(t_A;g_s)
 =\exp\sum_{g\ge2}g_s^{2g-2}U_g^A(t_A),
 \qquad F_g^{R,A,+}(Q)=a_g\operatorname{Li}_{3-2g}(Q).
 \end{gathered}
\end{equation}
The limit follows from the multiple-cover formula and agrees with Proposition~\ref{prop:gaussian-vacuum} after \eqref{eq:gw-phase}. It removes the finite-area regular terms. The A-model normalization of the stable neck tensors is
\begin{equation}\label{eq:gw-neck-a-tensor}
 \mathsf N^A_{h,n}(\mathbf m;t_A)
 =(-1)^{h-1}\mathcal A_{h,n}(\mathbf m;-it_A).
\end{equation}
These universal tensors couple to the background through the sewing coefficients $d_I,k_{IJ}$.

\subsection{The conifold frame and GW factorizations}
\label{subsec:gw-conventions}

We specify the conifold frame using the mirror curve of each target $Y$. Fix a continuation path in its moduli space from large radius to a smooth fiber near the conifold, and transport the large-radius symplectic basis $A,B$ of the compactified mirror curve. Write its oriented vanishing cycle as $A_c=\alpha A+\beta B$. On this curve, normalize the holomorphic differential by $\int_A\vartheta=1$ and put $\tau=\int_B\vartheta$. With the normalizations of Fang--Liu--Zong \cite[Section~5.1]{FLZ}, the transported large-radius kernel $B^{\mathrm{LR}}$ has $B$-period $2\pi i\vartheta$. The kernel $B^c$ with zero $A_c$-period is therefore
\begin{equation}\label{eq:gw-polarization-kernel}
 B^c-B^{\mathrm{LR}}
 =-\frac{2\pi i\beta}{\alpha+\beta\tau}\vartheta\otimes\vartheta.
\end{equation}

On the mirror curve, replacing $B^{\mathrm{LR}}$ by $B^c$ gives the connected stable-graph transformation of Eynard--Mari\~no--Orantin \cite[Section~4.2]{EMO}. Remodeling \eqref{eq:gw-remodeling-signs} transports this transformation to the analytically continued GW theory. With $t=-it_A$ and $\hbar=ig_s$, its vertices are the continued stable open and closed GW tensors in every genus. At each edge, integrate each incident variable over $B$ with factor $(2\pi i)^{-1}$ and multiply by $2\pi i\beta/(\alpha+\beta\tau)$, the negative of the scalar kernel correction in \eqref{eq:gw-polarization-kernel}. Denote the resulting GW potentials by $F_g^{Y,c,A}$ and put $\log Z_Y^{c,A}=\sum_{g\ge2}g_s^{2g-2}F_g^{Y,c,A}$. Thus
\begin{equation}\label{eq:gw-calibrated-comparison}
 Z_Y^{c,A}(t_A,s;g_s)
 =Z(\mathcal S^Y_{-it_A,s};ig_s).
\end{equation}

We now apply Theorem~\ref{thm:factorization} to this conifold-normalized theory through \eqref{eq:gw-calibrated-comparison}.

The spectral coefficients $I^Y_{g,k}=[t^k]\mathcal I_{Y,g}$ convert under \eqref{eq:gw-phase} to
\[
 I^{Y,A}_{g,k}=(-1)^{g-1}(-i)^kI^Y_{g,k}.
\]

\begin{corollary}[The first two GW factorizations]\label{gw:toric-factorization}
For the framings and conifold calibration specified above, and $g\ge2$, one has
\begin{align}
 F_g^{K_{\mathbb P^2},c,A}(t_A)
 &=C_gt_A^{2-2g}+3^{2-2g}A_g
       +\sum_{k\ge1}I_{g,k}^{\mathbb P^2,A}t_A^k,
 \label{eq:gw-p2-expansion}\\
 F_g^{K_{\mathbb F_0},c,A}(t_A,r)
 &=C_gt_A^{2-2g}
   +2^{2-2g}[2A_g+a_g\operatorname{Li}_{3-2g}(-r)]
       +\sum_{k\ge1}I_{g,k}^{\mathbb F_0,A}(r)t_A^k.
 \label{eq:gw-f0-expansion}
\end{align}
For local $\mathbb F_0$, the spectator parameter lies in $D_r$ from \eqref{eq:gw-comparison-domains}. Each coefficient of $t_A^k$, $k\ge1$, is a finite contraction of background open GW tensors and the relative-GW neck tensors, with source and edge weights $d_I,k_{IJ}$ supplied by the spectral family. For each fixed $g$, the remainders converge jointly near $t_A=0$ and the chosen spectator value.
\end{corollary}
\begin{proof}
Apply \eqref{eq:gw-calibrated-comparison} to \eqref{eq:gw-spectral-factorization}. The central background identifications and their scales give the constant terms by \eqref{eq:gw-background-vacua}. Proposition~\ref{prop:gw-neck-identification} and \eqref{eq:gw-universal-vacuum} identify the universal neck. Theorem~\ref{thm:factorization} gives the finite graph coefficients and the holomorphic remainder of positive period order. The substitutions \eqref{eq:gw-phase} preserve both conclusions.
\end{proof}

The finite constants in \eqref{eq:gw-p2-expansion} and \eqref{eq:gw-f0-expansion} agree with the matrix-model predictions of Mari\~no--Zakany \cite[Section~3.2]{MZ} and Kashaev--Mari\~no--Zakany \cite[Sections~3.1 and~3.3]{KMZ}, respectively. In the latter comparison, $r=e^{2\xi}$ in their notation, and the original degree-zero contribution $4A_g$ is added to their positive-degree potential. The background identifications above express these constants as the scaled $\mathbb C^3$ and resolved-conifold potentials.

\begin{corollary}[Local $\mathbb F_1$ factorization]\label{gw:f1-factorization}
For the chosen framing, conifold calibration, $q\in D_q$ from \eqref{eq:gw-comparison-domains}, and $g\ge2$,
\begin{equation}\label{eq:gw-f1-expansion}
 F_g^{K_{\mathbb F_1},c,A}(t_A,q)
 =C_gt_A^{2-2g}+F_g^{X_2,A}(-q(1-q)^3)
       +\sum_{k\ge1}I_{g,k}^{\mathbb F_1,A}(q)t_A^k.
\end{equation}
The coefficients of $t_A^k$, $k\ge1$, are finite contractions of the open $X_2$ tensors and the universal relative-GW neck tensors, with source and edge weights $d_I,k_{IJ}$ supplied by the spectral family. For each fixed $g$, the remainder is jointly holomorphic near $t_A=0$ and the chosen spectator value.
\end{corollary}
\begin{proof}
The original $K_{\mathbb F_1}$ mirror satisfies Lemma~\ref{lem:gw-regular-families}, and the remodeling theorem gives its calibrated comparison \eqref{eq:gw-calibrated-comparison}. Proposition~\ref{gw:f1} identifies the normalized background, whose scale is $1$. The rest follows from \eqref{eq:gw-spectral-factorization} exactly as above.
\end{proof}

For each of the three targets, define
\begin{equation}\label{eq:gw-interaction-conversion}
 \mathcal I_Y^A(t_A,s;g_s)=\mathcal I_Y(-it_A,s;ig_s).
\end{equation}
The resulting partition identities are
\begin{equation}\label{eq:gw-partition-factorization}
 Z_Y^{c,A}=Z_{\mathrm{back}}^A Z_{\mathrm{neck},Y}^A,
 \qquad Z_{\mathrm{neck},Y}^A
       =Z_{\mathrm{univ}}^A\exp\mathcal I_Y^A.
\end{equation}
Here $Z_{\mathrm{back}}^A$ is $Z_{\mathbb C^3}^A(g_s/3)$, $Z_R^A(-r;g_s/2)$, or $Z_{X_2}^A(-q(1-q)^3;g_s)$, respectively. These identities concern closed genera $g\ge2$.

\subsubsection*{The interaction term}
The series $\mathcal I_Y^A$ sums connected contractions of the background and neck tensors with the sewing sources and edges. At each genus $g\ge2$, its expansion in $t_A$ has positive order. We express these contractions in A-model terms.

Let $M$ denote a background target and put $\delta_M=-3,-2,1$ for $M=\mathbb C^3,R,X_2$, respectively. On a background coordinate disc, the A-model vertex is
\begin{equation}\label{eq:gw-background-tensor}
 \mathsf B^A_{h,n}(I_1,\ldots,I_n)
 =\delta_M^{2-2h-n}
 \left(\prod_{j=1}^n\mathcal C_{I_j}\right)
        d_1\cdots d_nF_{h,n}^{M,A}.
\end{equation}
The contours use the Morse coordinates of the inherited spectral data on $\mathcal B_s^Y$ and act after the open potential of $M$ has been continued to those discs. For the neck, the corresponding insertion tensors are precisely \eqref{eq:gw-neck-a-tensor}, with the two ends interpreted through \eqref{gau:end-moments}.

Let $\mathsf J=(\mathsf J_I)_I$ be formal variables indexed by the sewing contours $I=(a,m)$ at background ramification points or $I=(\pm,m)$ at neck ends. Extend each tensor by zero whenever an argument belongs to the other type, and set
\begin{equation}\label{eq:gw-vertices}
 \mathcal V^A(\mathsf J)=
 \sum_{\substack{h\ge0,n\ge1\\2h-2+n>0}}
 \frac{g_s^{2h-2+n}}{n!}
 \left(\mathsf B^A_{h,n}(\mathsf J^{\otimes n})
                 +\mathsf N^A_{h,n}(\mathsf J^{\otimes n};t_A)\right).
\end{equation}
With the sewing coefficients $d_I,k_{IJ}$ of \eqref{sew:couplings} evaluated at $t=-it_A$, the interaction is
\begin{equation}\label{eq:gw-graph-formula}
 \mathcal I_Y^A
 =\left[\log\left.
 \exp\left(-g_s^{-1}\sum_I d_I\partial_{\mathsf J_I}
       -\tfrac12\sum_{I,J}k_{IJ}
                    \partial_{\mathsf J_I}\partial_{\mathsf J_J}\right)
 e^{\mathcal V^A(\mathsf J)}\right|_{\mathsf J=0}\right]_{\ge2}.
\end{equation}
Here $[\cdot]_{\ge2}$ retains the terms $g_s^{2g-2}$ for $g\ge2$. This is \eqref{sew:graph-operator} under $\hbar=ig_s$, $J_I=-i\mathsf J_I$. Each source contributes $-d_I$, each edge contributes $-k_{IJ}$, and each connected graph carries its inverse automorphism order. Theorem~\ref{thm:factorization} supplies the coefficientwise finiteness and positive period order in \eqref{eq:gw-graph-formula}.

\subsection{Contour insertions as GW ancestors}
\label{subsec:gw-graphs}

The interaction formula \eqref{eq:gw-graph-formula} uses the open GW tensors \eqref{eq:gw-background-tensor}. We now express their contour insertions as closed GW ancestors.

For $M=\mathbb C^3,R,X_2$, write $\langle\!\langle\cdots\rangle\!\rangle^{M,A}_{h,n}$ for the connected closed equivariant GW ancestor correlator, summed over degrees and additional primary insertions at the same closed parameters and framing specialization as $F_{h,n}^{M,A}$. Its $n$ distinguished markings carry classes $\bar\psi_j$ pulled back from $\overline{\mathcal M}_{h,n}$ after forgetting the map and additional markings.

Let $\widehat e_a$ be the normalized quantum canonical basis, let $E$ have basis $\epsilon_a$, and define $\mathsf U:E\to H_T^*(M)$ by $\mathsf U\epsilon_a=\widehat e_a$, with equivariant coefficients localized and specialized. Use the calibrated GW matrix $\mathsf R(z)$ of \cite[Theorem~3.4]{FLZ} for $M=\mathbb C^3,R$ and \cite[Sections~1.2 and~2.3]{GXZ} for $M=X_2$, with its canonical-basis indices regarded as indices on $E$. It satisfies $\mathsf R(-z)^T\mathsf R(z)=I$. For $R$, use $Q'_R=-1/r$; for $X_2$, use Proposition~\ref{gw:f1}. Match ramification points with quantum idempotents and choose square roots as in \cite[Section~7.2]{FLZ} and \cite[equations~(29)--(31)]{GXZ}, respectively. Use Morse coordinates $x-x(a)=\xi_a^2/2$.

\Needspace{8\baselineskip}
\begin{proposition}[The ancestor insertion]\label{prop:gw-ancestor-insertion}
For a background contour $\mathcal C_\phi\omega=(2\pi i)^{-1}\int_{\partial D_a}\phi\omega$ with holomorphic weight $\phi$, define $s_\phi(z)$ by the following formula. For the weight $\xi_a^m/m$ of $\mathcal C_{a,m}$, write $s_{a,m}=s_\phi$; then
\begin{align}
 s_\phi(z)&=-i\sum_{\ell\ge0}(2\ell+1)!!\,[\xi_a^{2\ell+1}]\phi
                  \;z^\ell\mathsf U\mathsf R(z)\epsilon_a,
 \label{eq:gw-ancestor-state}\\
 s_{a,2\ell+1}(z)&=-i(2\ell-1)!!z^\ell\mathsf U\mathsf R(z)\epsilon_a,
 \qquad s_{a,2\ell+2}(z)=0.
 \label{eq:gw-odd-ancestor-state}
\end{align}
These are cohomology-valued formal series, with $\ell\ge0$ and $(-1)!!=1$. For $M=\mathbb C^3,R,X_2$, $n\ge1$, and $2h-2+n>0$, one has
\begin{equation}\label{eq:gw-contour-ancestor-correlator}
 \left(\prod_{j=1}^n\mathcal C_{\phi_j}\right)
 d_1\cdots d_nF_{h,n}^{M,A}
 =\left\langle\!\left\langle
 s_{\phi_1}(\bar\psi_1),\ldots,s_{\phi_n}(\bar\psi_n)
 \right\rangle\!\right\rangle^{M,A}_{h,n}.
\end{equation}
Thus each contour on an open GW slot gives an ancestor insertion at the corresponding closed marking. The background tensor \eqref{eq:gw-background-tensor} is the ancestor correlator multiplied by $\delta_M^{2-2h-n}$. Only finitely many ancestor powers contribute at fixed $(h,n)$.
\end{proposition}
\begin{proof}
For $M=\mathbb C^3,R$, the differentials in the B-model graph expansion have principal parts
\begin{equation}\label{eq:gw-leaf-principal-part}
 \theta_a^\ell=-\sqrt2(2\ell+1)!!\,
                 \frac{d\xi_a}{\xi_a^{2\ell+2}}+\text{holomorphic},
\end{equation}
and the differentiated A-model open leaf is $-\theta_a^\ell/\sqrt{-2}$ \cite[Sections~5.10 and~7.3]{FLZ}. With $\sqrt{-2}=i\sqrt2$, contour integration extracts $-i(2\ell+1)!![\xi_a^{2\ell+1}]\phi$. The descendant--ancestor transformation of Kontsevich--Manin \cite{KM}, in the quantized form of \cite[Theorem~5.1]{Givental}, removes the $S$-operator from the descendant leaf in \cite[equations~(3.10)--(3.12)]{FLZ}. An ancestor insertion $s(z)$ therefore contributes the coordinate column $\mathsf R(-z)^T\mathsf U^{-1}s(z)$; equivalently, this is $\mathsf R(z)^{-1}\mathsf U^{-1}s(z)$ in \cite[Proposition~7.3]{Givental}. For $s=s_\phi$, unitarity gives the same extracted coefficients. The remaining graph weights agree under remodeling, proving \eqref{eq:gw-contour-ancestor-correlator}. Taking $\phi=\xi_a^m/m$ gives \eqref{eq:gw-odd-ancestor-state}. Finiteness follows from $\dim\overline{\mathcal M}_{h,n}=3h-3+n$.

For $M=X_2$, \cite[Section~1.4 and Propositions~2.2 and~2.5]{GXZ} identifies the closed ancestor graph sum and its open leaves. By their Laplace formula \cite[equations~(22) and~(50)]{GXZ}, the inverse $R$-matrix cancels the lower polar terms in the full open leaf. Apply Proposition~\ref{gw:f1} and take $\eta_*=i\xi_a$ for their Morse coordinate. The open-coordinate and normal-weight changes give one minus sign per boundary, giving the same extracted coefficient $-i(2\ell+1)!![\xi_a^{2\ell+1}]\phi$ and ancestor leaf $\mathsf R(z)^{-1}\mathsf U^{-1}s(z)$. The preceding calculation then applies.
\end{proof}

For all three targets, Propositions~\ref{prop:gw-ancestor-insertion} and~\ref{prop:gw-neck-identification} express each positive-power coefficient in \eqref{eq:gw-graph-formula} as a finite contraction of closed background ancestors and relative GW invariants of $\mathbb P^1$, with the scales \eqref{eq:gw-background-tensor} and spectral sewing weights \eqref{sew:couplings}. The Gaussian vacuum supplies the universal pole, the normalized background supplies the finite part, and these contractions determine the remaining Taylor coefficients.


\begin{thebibliography}{ABDKS26}
\raggedright

\bibitem[ABDKS26]{ABDKS}
A.~Alexandrov, B.~Bychkov, P.~Dunin-Barkowski, M.~Kazarian, and S.~Shadrin, \emph{Any topological recursion on a rational spectral curve is KP integrable}, Commun. Math. Phys. \textbf{407} (2026), article~69, \href{https://doi.org/10.1007/s00220-026-05566-9}{doi:10.1007/s00220-026-05566-9}.

\bibitem[ACNP15]{ACNP}
J.~E.~Andersen, L.~O.~Chekhov, P.~Norbury, and R.~C.~Penner, \emph{Topological recursion for Gaussian means and cohomological field theories}, Theor. Math. Phys. \textbf{185} (2015), 1685--1717, \href{https://doi.org/10.1007/s11232-015-0373-0}{doi:10.1007/s11232-015-0373-0}.

\bibitem[Ali25]{Alim}
M.~Alim, \emph{Intrinsic non-perturbative topological strings}, Adv. Theor. Math. Phys. \textbf{29} (2025), no.~5, 1365--1406, \href{https://doi.org/10.4310/atmp.251023013812}{doi:10.4310/atmp.251023013812}.

\bibitem[BCOV94]{BCOV}
M.~Bershadsky, S.~Cecotti, H.~Ooguri, and C.~Vafa, \emph{Kodaira--Spencer theory of gravity and exact results for quantum string amplitudes}, Commun. Math. Phys. \textbf{165} (1994), 311--428, \href{https://doi.org/10.1007/BF02099774}{doi:10.1007/BF02099774}.

\bibitem[BKMP09]{BKMP}
V.~Bouchard, A.~Klemm, M.~Mari\~no, and S.~Pasquetti, \emph{Remodeling the B-model}, Commun. Math. Phys. \textbf{287} (2009), 117--178, \href{https://doi.org/10.1007/s00220-008-0620-4}{doi:10.1007/s00220-008-0620-4}.

\bibitem[BKMP10]{BKMPorb}
V.~Bouchard, A.~Klemm, M.~Mari\~no, and S.~Pasquetti, \emph{Topological open strings on orbifolds}, Commun. Math. Phys. \textbf{296} (2010), 589--623, \href{https://doi.org/10.1007/s00220-010-1020-0}{doi:10.1007/s00220-010-1020-0}.

\bibitem[Bri25]{Brini}
A.~Brini, \emph{Conifold gap and all-genus mirror symmetry for local $\mathbb P^2$}, \href{https://arxiv.org/abs/2509.19298}{arXiv:2509.19298}.

\bibitem[CGYZ26]{CGYZ}
H.-L.~Chang, S.~Guo, L.~You, and H.~Zhang, \emph{Conifold gap for the quintic threefold}, \href{https://arxiv.org/abs/2609.13921}{arXiv:2609.13921}.

\bibitem[CCFZ26]{CCFZ}
J.~Chen, L.~Chen, B.~Fang, and Z.~Zong, \emph{Conifold Gap Theorem for Topological Recursion}, \href{https://arxiv.org/abs/2608.11960v1}{arXiv:2608.11960v1}.

\bibitem[CI21]{CI}
T.~Coates and H.~Iritani, \emph{Gromov--Witten invariants of local $\mathbb P^2$ and modular forms}, Kyoto J. Math. \textbf{61} (2021), no.~3, 543--706, \href{https://doi.org/10.1215/21562261-2021-0010}{doi:10.1215/21562261-2021-0010}.

\bibitem[dMKR10]{dMKR}
R.~de Mello Koch and S.~Ramgoolam, \emph{From Matrix Models and quantum fields to Hurwitz space and the absolute Galois group}, \href{https://arxiv.org/abs/1002.1634}{arXiv:1002.1634}.

\bibitem[DMSS13]{DMSS}
O.~Dumitrescu, M.~Mulase, B.~Safnuk, and A.~Sorkin, \emph{The spectral curve of the Eynard--Orantin recursion via the Laplace transform}, in \emph{Algebraic and Geometric Aspects of Integrable Systems and Random Matrices}, Contemp. Math. \textbf{593} (2013), 263--315, \href{https://doi.org/10.1090/conm/593/11867}{doi:10.1090/conm/593/11867}.

\bibitem[DOSS14]{DOSS}
P.~Dunin-Barkowski, N.~Orantin, S.~Shadrin, and L.~Spitz, \emph{Identification of the Givental formula with the spectral curve topological recursion procedure}, Commun. Math. Phys. \textbf{328} (2014), 669--700, \href{https://doi.org/10.1007/s00220-014-1887-2}{doi:10.1007/s00220-014-1887-2}.

\bibitem[Eyn08]{EynardPartitions}
B.~Eynard, \emph{All orders asymptotic expansion of large partitions}, J. Stat. Mech. Theory Exp. \textbf{2008} (2008), no.~7, article~P07023, \href{https://doi.org/10.1088/1742-5468/2008/07/P07023}{doi:10.1088/1742-5468/2008/07/P07023}.

\bibitem[Eyn14]{EynardIntersection}
B.~Eynard, \emph{Invariants of spectral curves and intersection theory of moduli spaces of complex curves}, Commun. Number Theory Phys. \textbf{8} (2014), no.~3, 541--588, \href{https://doi.org/10.4310/CNTP.2014.v8.n3.a4}{doi:10.4310/CNTP.2014.v8.n3.a4}.

\bibitem[EMO07]{EMO}
B.~Eynard, M.~Mari\~no, and N.~Orantin, \emph{Holomorphic anomaly and matrix models}, J. High Energy Phys. \textbf{2007} (2007), no.~6, article~058, \href{https://doi.org/10.1088/1126-6708/2007/06/058}{doi:10.1088/1126-6708/2007/06/058}.

\bibitem[EO07]{EO2007}
B.~Eynard and N.~Orantin, \emph{Invariants of algebraic curves and topological expansion}, Commun. Number Theory Phys. \textbf{1} (2007), 347--452, \href{https://doi.org/10.4310/CNTP.2007.v1.n2.a4}{doi:10.4310/CNTP.2007.v1.n2.a4}.

\bibitem[EO15]{EO2015}
B.~Eynard and N.~Orantin, \emph{Computation of open Gromov--Witten invariants for toric Calabi--Yau $3$-folds by topological recursion, a proof of the BKMP conjecture}, Commun. Math. Phys. \textbf{337} (2015), 483--567, \href{https://doi.org/10.1007/s00220-015-2361-5}{doi:10.1007/s00220-015-2361-5}.

\bibitem[FP00]{FP}
C.~Faber and R.~Pandharipande, \emph{Hodge integrals and Gromov--Witten theory}, Invent. Math. \textbf{139} (2000), 173--199, \href{https://doi.org/10.1007/s002229900028}{doi:10.1007/s002229900028}.

\bibitem[FL13]{FangLiu}
B.~Fang and C.-C.~M.~Liu, \emph{Open Gromov--Witten invariants of toric Calabi--Yau $3$-folds}, Commun. Math. Phys. \textbf{323} (2013), 285--328, \href{https://doi.org/10.1007/s00220-013-1771-5}{doi:10.1007/s00220-013-1771-5}.

\bibitem[FLZ20]{FLZ}
B.~Fang, C.-C.~M.~Liu, and Z.~Zong, \emph{On the remodeling conjecture for toric Calabi--Yau $3$-orbifolds}, J. Amer. Math. Soc. \textbf{33} (2020), 135--222, \href{https://doi.org/10.1090/jams/934}{doi:10.1090/jams/934}.

\bibitem[GV95]{GhoshalVafa}
D.~Ghoshal and C.~Vafa, \emph{$c=1$ string as the topological theory of the conifold}, Nuclear Phys. B \textbf{453} (1995), 121--128, \href{https://doi.org/10.1016/0550-3213(95)00408-K}{doi:10.1016/0550-3213(95)00408-K}.

\bibitem[Giv01]{Givental}
A.~B.~Givental, \emph{Gromov--Witten invariants and quantization of quadratic Hamiltonians}, Mosc. Math. J. \textbf{1} (2001), no.~4, 551--568, \href{https://doi.org/10.17323/1609-4514-2001-1-4-551-568}{doi:10.17323/1609-4514-2001-1-4-551-568}.

\bibitem[GHM16]{GHM}
A.~Grassi, Y.~Hatsuda, and M.~Mari\~no, \emph{Topological strings from quantum mechanics}, Ann. Henri Poincar\'e \textbf{17} (2016), 3177--3235, \href{https://doi.org/10.1007/s00023-016-0479-4}{doi:10.1007/s00023-016-0479-4}.

\bibitem[GKN19]{GKN}
S.~Grushevsky, I.~Krichever, and C.~Norton, \emph{Real-normalized differentials: limits on stable curves}, Russian Math. Surveys \textbf{74} (2019), no.~2, 265--324, \href{https://doi.org/10.1070/RM9877}{doi:10.1070/RM9877}.

\bibitem[GXZ26]{GXZ}
S.~Guo, J.~Xu, and Q.~Zhang, \emph{Structure of higher-genus open-closed Gromov--Witten theory of $\mathcal O_{\mathbb P^1}(p-1)\oplus\mathcal O_{\mathbb P^1}(-p-1)$}, \href{https://arxiv.org/abs/2609.40257}{arXiv:2609.40257}.

\bibitem[HZ86]{HZ}
J.~Harer and D.~Zagier, \emph{The Euler characteristic of the moduli space of curves}, Invent. Math. \textbf{85} (1986), 457--485, \href{https://doi.org/10.1007/BF01390325}{doi:10.1007/BF01390325}.

\bibitem[HN20]{HuNorton}
X.~Hu and C.~Norton, \emph{General variational formulas for Abelian differentials}, Int. Math. Res. Not. (2020), no.~12, 3540--3581, \href{https://doi.org/10.1093/imrn/rny106}{doi:10.1093/imrn/rny106}.

\bibitem[HK07]{HK}
M.-x.~Huang and A.~Klemm, \emph{Holomorphic anomaly in gauge theories and matrix models}, J. High Energy Phys. \textbf{2007} (2007), no.~9, article~054, \href{https://doi.org/10.1088/1126-6708/2007/09/054}{doi:10.1088/1126-6708/2007/09/054}.

\bibitem[HKQ09]{HKQ}
M.-x.~Huang, A.~Klemm, and S.~Quackenbush, \emph{Topological string theory on compact Calabi--Yau: modularity and boundary conditions}, Lecture Notes in Phys. \textbf{757} (2009), 45--102, \href{https://doi.org/10.1007/978-3-540-68030-7_3}{doi:10.1007/978-3-540-68030-7\_3}.

\bibitem[IILZ25]{IILZ}
N.~Iorgov, K.~Iwaki, O.~Lisovyy, and Y.~Zhuravlov, \emph{Many-faced Painlev\'e~I: irregular conformal blocks, topological recursion, and holomorphic anomaly approaches}, \href{https://arxiv.org/abs/2505.16803}{arXiv:2505.16803}.

\bibitem[IK22]{IK}
K.~Iwaki and O.~Kidwai, \emph{Topological recursion and uncoupled BPS structures I: BPS spectrum and free energies}, Adv. Math. \textbf{398} (2022), article~108191, \href{https://doi.org/10.1016/j.aim.2022.108191}{doi:10.1016/j.aim.2022.108191}.

\bibitem[IKT23]{IKT}
K.~Iwaki, T.~Koike, and Y.~Takei, \emph{Voros coefficients for the hypergeometric differential equations and Eynard--Orantin's topological recursion: Part~I---For the Weber equation}, Ann. Henri Poincar\'e \textbf{24} (2023), 1305--1353, \href{https://doi.org/10.1007/s00023-022-01235-4}{doi:10.1007/s00023-022-01235-4}.

\bibitem[KMZ16]{KMZ}
R.~Kashaev, M.~Mari\~no, and S.~Zakany, \emph{Matrix models from operators and topological strings, 2}, Ann. Henri Poincar\'e \textbf{17} (2016), 2741--2781, \href{https://doi.org/10.1007/s00023-016-0471-z}{doi:10.1007/s00023-016-0471-z}.

\bibitem[KM98]{KM}
M.~Kontsevich and Y.~I.~Manin, \emph{Relations between the correlators of the topological sigma-model coupled to gravity}, Commun. Math. Phys. \textbf{196} (1998), 385--398, \href{https://doi.org/10.1007/s002200050426}{doi:10.1007/s002200050426}.

\bibitem[LLLZ09]{LLLZ}
J.~Li, C.-C.~M.~Liu, K.~Liu, and J.~Zhou, \emph{A mathematical theory of the topological vertex}, Geom. Topol. \textbf{13} (2009), 527--621, \href{https://doi.org/10.2140/gt.2009.13.527}{doi:10.2140/gt.2009.13.527}.

\bibitem[Mar08]{MarinoLocalCurves}
M.~Mari\~no, \emph{Open string amplitudes and large order behavior in topological string theory}, J. High Energy Phys. \textbf{2008} (2008), no.~3, article~060, \href{https://doi.org/10.1088/1126-6708/2008/03/060}{doi:10.1088/1126-6708/2008/03/060}.

\bibitem[MZ16]{MZ}
M.~Mari\~no and S.~Zakany, \emph{Matrix models from operators and topological strings}, Ann. Henri Poincar\'e \textbf{17} (2016), 1075--1108, \href{https://doi.org/10.1007/s00023-015-0422-0}{doi:10.1007/s00023-015-0422-0}.

\bibitem[OP06]{OP}
A.~Okounkov and R.~Pandharipande, \emph{Gromov--Witten theory, Hurwitz theory, and completed cycles}, Ann. of Math. \textbf{163} (2006), 517--560, \href{https://doi.org/10.4007/annals.2006.163.517}{doi:10.4007/annals.2006.163.517}.

\bibitem[PS10]{PS}
S.~Pasquetti and R.~Schiappa, \emph{Borel and Stokes nonperturbative phenomena in topological string theory and $c=1$ matrix models}, Ann. Henri Poincar\'e \textbf{11} (2010), 351--431, \href{https://doi.org/10.1007/s00023-010-0044-5}{doi:10.1007/s00023-010-0044-5}.

\bibitem[Str95]{Strominger}
A.~Strominger, \emph{Massless black holes and conifolds in string theory}, Nuclear Phys. B \textbf{451} (1995), 96--108, \href{https://doi.org/10.1016/0550-3213(95)00287-3}{doi:10.1016/0550-3213(95)00287-3}.

\bibitem[Yam80]{Yamada}
A.~Yamada, \emph{Precise variational formulas for abelian differentials}, Kodai Math. J. \textbf{3} (1980), 114--143, \href{https://doi.org/10.2996/kmj/1138036124}{doi:10.2996/kmj/1138036124}.

\end{thebibliography}
\end{document}